\documentclass[a4paper,11pt,reqno]{amsart}

\usepackage[utf8]{inputenc}
\usepackage{amsmath, amsthm, amssymb}
\usepackage[svgnames]{xcolor}
\usepackage{microtype}
\usepackage{graphicx}
\usepackage{xspace}
\usepackage{setspace}
\usepackage{stmaryrd}
\usepackage{mathtools}
\usepackage{tikz}
\usetikzlibrary{shapes}
\usepackage{a4wide}
\usepackage{nicefrac}
\usepackage{nameref}

\allowdisplaybreaks

\newtheorem{thm}{Theorem}[section]
\newtheorem{lem}[thm]{Lemma}
\newtheorem{sublem}[thm]{Sublemma}
\newtheorem{prop}[thm]{Proposition}

\newtheorem{corollary}[thm]{Corollary}
\newtheorem{question}[thm]{Question}
\newtheorem{defn}[thm]{Definition}

\theoremstyle{remark}
\newtheorem{remark}[thm]{Remark}
\newtheorem{rmks}[thm]{Remarks}

\newcommand{\A}{\mathcal{A}}
\newcommand{\B}{\mathcal{B}}
\newcommand{\C}{\mathcal{C}}
\newcommand{\D}{\mathcal{D}}
\newcommand{\E}{\mathcal{E}}

\newcommand{\N}{\mathcal{N}}

\renewcommand{\S}{\mathcal{S}}

\newcommand{\U}{\mathcal{U}}
\newcommand{\V}{\mathcal{V}}
\newcommand{\X}{\mathcal{X}}

\newcommand{\MLR}{\mathrm{MLR}}
\newcommand{\COMP}{\mathrm{COMP}}
\newcommand{\IM}{\mathrm{IM}}
\newcommand{\BI}{\mathrm{BI}}
\newcommand{\HI}{\mathrm{HI}}
\newcommand{\BHI}{\mathrm{BHI}}
\newcommand{\SR}{\mathrm{SR}}

\newcommand{\KR}{\mathrm{KR}}

\newcommand{\dom}{\mathrm{dom}}
\newcommand{\LV}{\mathrm{LV}}

\newcommand{\cs}{2^\omega}
\newcommand{\uh}{{\upharpoonright}}
\renewcommand{\phi}{\varphi}
\newcommand{\str}{2^{<\omega}}
\newcommand{\binrat}{\mathbb{Q}_2}

\newcommand{\halts}{{\downarrow}}

\newcommand{\task}{\mathfrak{t}} 
\newcommand{\ct}{\mathfrak{c}}
\newcommand{\llb}{\llbracket}
\newcommand{\rrb}{\rrbracket}

\newcommand{\lla}{\langle}
\newcommand{\rra}{\rangle}

\def\supp#1{\mathrm{Supp}(#1)}
\newcommand{\dspec}{\mathrm{Spec}}
\def\T{\mathrm{T}}

\newcommand{\emptystr}{\varepsilon}

\def\fan{\mathrm{Fan}}

\def\qt#1{``#1''}

\title{Degrees of Randomized Computability}
\author{Rupert H\"olzl}
\address{Institut~1,
	Fakultät für Informatik,
	Universität der Bundeswehr München,
	Werner-Heisenberg-Weg~39,
	85579~Neubiberg,
	Germany}
\email{r@hoelzl.fr}
\urladdr{http://hoelzl.fr}

\author{Christopher P. Porter}
\address{Department of Mathematics and Computer Science,
	Drake University,
	Des Moines, IA 50311, 
	USA}
\email{cp@cpporter.com}
\urladdr{http://cpporter.com}

\begin{document}

\maketitle

\begin{abstract}
In this survey we discuss work of Levin and V'yugin on collections of sequences that are non-negligible in the sense that they can be computed by a probabilistic algorithm with positive probability. More precisely, Levin and V'yugin introduced an ordering on collections of sequences that are closed under Turing equivalence.  Roughly speaking, given two such collections $\mathcal{A}$ and $\mathcal{B}$, $\mathcal{A}$~is below $\mathcal{B}$ in this ordering if $\mathcal{A}\setminus\mathcal{B}$ is negligible. The degree structure associated with this ordering, the \textit{Levin-V'yugin degrees} (or \textit{$\LV$-degrees}), can be shown to be a Boolean algebra, and in fact a measure algebra.
	
We demonstrate the interactions of this work with recent results in computability theory and algorithmic randomness: First, we recall the definition of the Levin-V'yugin algebra and identify connections between its properties and classical properties from computability theory.  In particular,  we apply results on the interactions between notions of randomness and Turing reducibility to establish new facts about specific LV-degrees, such as the LV-degree of the collection of 1-generic sequences, that of the collection of sequences of hyperimmune degree, and those collections corresponding to  various notions of effective randomness. Next, we provide a detailed explanation of a complex technique developed by V'yugin that allows the construction of semi-measures into which computability-theoretic properties can be encoded. We provide two examples of the use of this technique by explicating a result of V'yugin's about the LV-degree of the collection of Martin-L\"of random sequences and extending the result to the LV-degree of the collection of sequences of DNC~degree.
\end{abstract}

\tableofcontents

\section{Introduction}

The tools of algorithmic randomness have been particularly useful in studying the power of random oracles in the context of Turing reducibility.  It is well-known that access to a random oracle does not aid in the computation of any \emph{individual} sequence, as Sacks~\cite{Sac63} proved that any sequence that is computable from positive measure many oracles must be computable. However, if instead we attempt to compute some element of a collection of sequences by means of a random oracle, the situation is quite different.

For instance, in unpublished work, Martin proved that the collection of sequences of hyperimmune degree has Lebesgue measure $1$ (see Downey and Hirschfeldt~\cite[Theorem~8.21.1]{DowHir10}). A~careful examination of this proof yields, for any $\delta\in(0,1)$, an algorithm which with probability at least $1-\delta$ computes from a random oracle a function not dominated by any computable function (as~noted by G\'acs  and reported by Rumyantsev and Shen~\cite{RumShe13}).  Other types of sequences known to be computable from positive measure many sequences are the 1-generic sequences (as~shown by Kurtz~\cite{Kur81} and  Kautz~\cite{Kau91}), the sequences of DNC degree (first established by Ku\v cera~\cite{Kuc85}), and sequences satisfying certain algebraic properties in the upper semi-lattice of the Turing degrees under Turing reducibility (studied by Barmpalias, Day, and Lewis-Pye \cite{BarDayLew14}).

Collections of sequences $\C\subseteq\cs$ with the property that only  measure $0$ many sequences compute an element of $\C$ have been referred to as \emph{negligible} (for instance, in V'yugin \cite{Vyu82} and Levin \cite{Lev84}), and thus those collections $\C$ with the property that positive measure many sequences compute an element of $\C$ are called \emph{non-negligible}.  
The focus of our study here is a Boolean algebra of non-negligible subsets of $\cs$ that are closed under Turing equivalence and where two such subsets are identified with each other if they differ only by a negligible set. 
This Boolean algebra, first introduced by Levin and V'yugin \cite{LevVyu77} and systematically studied by V'yugin~\cite{Vyu82}, will be referred to as the \emph{Levin-V'yugin algebra}; its elements will be referred to as the \emph{Levin-V'yugin degrees}, or $\LV$\emph{-degrees} for short.

A significant portion of this article is a survey of previously established results about the Levin-V'yugin algebra, but we also establish 
new facts about it as well.  Much of our focus will furthermore be on explicating a technique developed by V'yugin~\cite{Vyu82} for building left-c.e.\ semi-measures, 
which has applications outside of the study of the algebra, such as in the study of probabilistic computation. We first provide a general schematic account of this technique and then use it to establish the following result.
\begin{thm}[V'yugin  \cite{Vyu12}]\label{thm-prob-alg}
For any $\delta\in(0,1)$, there is a probabilistic algorithm that produces with probability at least $1-\delta$ a non-computable sequence that does not compute any Martin-L\"of random sequence.
\end{thm}
We will then apply V'yugin's technique to prove the following generalization of Theorem~\ref{thm-prob-alg}.

\begin{thm}\label{thm-prob-alg2}
For any $\delta\in(0,1)$, there is a probabilistic algorithm  that produces  with probability at least~$1-\delta$  a non-computable sequence that is not of DNC degree.
\end{thm}

Theorems \ref{thm-prob-alg} and \ref{thm-prob-alg2} both follow from a result due to Kurtz \cite{Kur81}, namely that for every~$\delta\in(0,1)$, there is a probabilistic algorithm that produces a $1$-generic sequence with probability $1-\delta$.  Since a $1$-generic sequence can compute neither a Martin-L\"of random sequence nor  a sequence of DNC degree, the results follow. However, V'yugin's technique also has implications for the study of $\Pi^0_1$~classes, that is, effectively closed subsets of $\cs$: the probabilistic algorithms whose existence can be shown using V'yugin's technique are in fact Turing functionals on~$\cs$ with a closed range; and since such a functional is effective, its range is even $\Pi^0_1$. Thus, V'yugin's proof of Theorem \ref{thm-prob-alg} establishes the following stronger result.

\begin{corollary}\label{cor1}{\ }
For every $\delta\in(0,1)$, there is a Turing functional $\Phi$ such that 
\begin{itemize}
\item[(i)] $\Phi$ maps no set of positive measure to any single sequence,
\item [(ii)] the domain of $\Phi$ has Lebesgue measure at least $1-\delta$,
\item [(iii)] the range of $\Phi$ is a $\Pi^0_1$ class, and
\item [(iv)] no sequence in the range of $\Phi$ computes a Martin-L\"of random sequence.
\end{itemize}

\end{corollary}

Similarly, the proof of Theorem \ref{thm-prob-alg2} that we provide here establishes the following result.

\begin{corollary}\label{cor2}{\ }
For every $\delta\in(0,1)$, there is a Turing functional $\Phi$ such that 
\begin{itemize}
\item[(i)] $\Phi$ maps no set of positive measure to any single sequence,
\item [(ii)] the domain of $\Phi$ has Lebesgue measure at least $1-\delta$,
\item [(iii)] the range of $\Phi$ is a $\Pi^0_1$ class, and
\item [(iv)] no sequence in the range of $\Phi$ is of DNC degree.
\end{itemize}
\end{corollary}

The remainder of this article is structured as follows. In Section \ref{sec-background}, we review the necessary background.   Section \ref{sec-neg-nonneg} introduces the notions of negligibility and non-negligibility and provides a number of examples from classical computability theory and algorithmic randomness.  The Levin-V'yugin degrees, defined in terms of negligibility, are introduced in Section \ref{sec-lv-degrees}.  
The general features of V'yugin's technique for constructing semi-measures are initially laid out in Section \ref{sec-building}, while specific examples of the technique are provided in Section \ref{sec-implementing}. Lastly, in Section~\ref{sec-conclusion} we conclude with a final observation about the connection between V'yugin's technique and $\Pi^0_1$ classes.

\section{Background}\label{sec-background}

\subsection{Some notation}

 We fix the following notation and terminology. We denote  the natural numbers by~$\omega$, and the set of infinite binary sequences, also known as \emph{Cantor space}, by $\cs$. We denote the set of finite binary strings by~$\str$ and the empty string by~$\emptystr$. Let $\binrat$ be the set of non-negative dyadic rationals, that is, rationals of the form $m/2^n$ for $m,n\in\omega$.  
 
 Given $X\in \cs$ and an integer~$n$, $X \uh n$ is the string that consists of the first~$n$ bits of~$X$, and $X(n)$~is the $(n+1)^{\mathrm{st}}$ bit of $X$ (so that $X(0)$ is the first bit of $X$).  If $\sigma$ and $\tau$ are strings, then $\sigma\preceq\tau$ means that $\sigma$ is an initial segment of $\tau$.  Similarly, for $X\in\cs$, $\sigma\prec X$ means that $\sigma$ is an initial segment of~$X$.  
 
 Given a string~$\sigma$, the \emph{cylinder} $\llb\sigma\rrb$ is the collection of elements of $\cs$ having~$\sigma$ as an initial segment.  Similarly, given $S\subseteq\str$, $\llb S\rrb$ is defined to be the collection $\bigcup_{\sigma\in S}\llb\sigma\rrb$. The cylinders form a basis for the usual product topology on Cantor space, and thus the open sets for this topology are those of the form $\llb S \rrb$ for some~$S$. An open set $\mathcal{U}$ is said to be \emph{effectively open}~(or~$\Sigma^0_1$) if $\mathcal{U}=\llb S \rrb$ for some computably enumerable (hereafter, c.e.) set~$S \subseteq \str$. An \emph{effectively closed} (or~$\Pi^0_1$) set is the complement of an effectively open set. A sequence of open sets $(\mathcal{U}_n)_{n \in \omega}$ is said to be \emph{uniformly effectively  open} if there exists a sequence $(S_n)_{n \in \omega}$ of uniformly c.e.\ sets of strings such that $\mathcal{U}_n=\llb S_n \rrb$ for all $n\in\omega$. 

For $\A\subseteq\cs$, we write $(\A)^{\equiv_\T }$ for the closure of $\A$ under Turing equivalence; that is, we let
\[(\A)^{\equiv_\T }:=\{X\in\cs\colon (\exists Y\in\A)\;X\equiv_\T  Y\}.\]

\subsection{Turing functionals and computable measures}  We assume that the reader is familiar with the basics of computability theory (for instance, the material covered in Soare~\cite[Chapters~I\nobreakdash-IV]{Soa16}, Nies~\cite[Chapter~1]{Nie09}, or Downey and Hirschfeldt~\cite[Chapter~2]{DowHir10}).

\begin{defn}\
\begin{itemize}
\item[(i)] A \emph{Turing functional} $\Phi\colon\subseteq\cs\rightarrow\cs$ is represented by a c.e.\ set $S_\Phi$ of pairs of strings $(\sigma,\tau)$ such that 
if $(\sigma,\tau),(\sigma',\tau')\in S_\Phi$ and $\sigma\preceq\sigma'$, then $\tau\preceq\tau'$ or $\tau'\preceq\tau$.  
\item[(ii)] For each $\sigma\in\str$, we define $\Phi^\sigma$ to be the maximal string (in the order given by~$\preceq$) in the set~$\{\tau\colon (\exists \sigma'\preceq\sigma)((\sigma',\tau)\in S_\Phi)\} \cup \{\varepsilon\}$.  
Similarly, for each $s\in\omega$, $\Phi^\sigma_s$ is the maximal string in the set
${\{\tau\colon (\exists \sigma'\preceq\sigma)((\sigma',\tau)\in S_\Phi[s])\} \cup \{\varepsilon\}}$, where $S_\Phi[s]$ is the  approximation of the c.e.~set~$S_\Phi$ at stage~$s$.
\item[(iii)] Let $\Phi^X$ be the minimal (in the order given by $\preceq$) $z\in\str\cup\cs$ such that $\Phi^{X \uh n} \preceq z$ for all~$n$.  
\item[(iv)]  We set $\dom(\Phi)=\{X\in\cs\colon\Phi^X\in\cs\}$.  
\item[(v)] For $\tau \in \str$, let $\Phi^{-1}(\tau)$ be 
$\{ \sigma\in \str \colon \exists \tau' \succeq \tau \colon (\sigma,\tau')\in S_\Phi\}$. 
\item[(vi)] Lastly, for $\A \subseteq \cs$, let $\Phi^{-1}(\A)$ be 
$\{
X \in \cs\colon \Phi^X \in \A
\}$.
\end{itemize}
\end{defn}

When $\Phi^X\in\cs$, we will often write $\Phi^X$ as $\Phi(X)$ to emphasize that we view the functional~$\Phi$ as a (partial) map from~$\cs$ to~$\cs$. 

A measure $\mu$ on $\cs$ is \emph{computable} if $\sigma \mapsto \mu(\llb\sigma\rrb)$ is computable as a real-valued function, that~is, if there is a computable function $\widetilde \mu\colon\str\times\omega\rightarrow\binrat$ such that
\[|\mu(\llb\sigma\rrb)-\widetilde \mu(\sigma,i)|\leq 2^{-i}\]
for every $\sigma\in\str$ and $i\in\omega$. For all measures appearing in this article we assume that $\mu(\cs)\leq 1$ without explicit mention.
From now on, we will write $\mu(\llb\sigma\rrb)$ as $\mu(\sigma)$.  By Carathéodory's Theorem, if the values $\mu(\sigma)$, for $\sigma\in\str$, of a measure $\mu$ on $\cs$ are fixed, then there is a unique extension of~$\mu$ to the Borel $\sigma$-algebra generated by the sets $\llb \sigma \rrb$, for $\sigma\in\str$. In this article, all measures will be defined in this way, which implies in particular that the same sets are measurable for each of these measures.

The \emph{uniform (or Lebesgue) measure}~$\lambda$ is the probability measure for which each bit of the sequence has value~$0$ with probability $1/2$, independently of the values of the other bits. It can be defined as the unique Borel measure such that $\lambda(\sigma)=2^{-|\sigma|}$ for all strings~$\sigma$. Clearly, $\lambda$~is a computable measure.

\subsection{Notions of algorithmic randomness}\label{subsec-notions}

The primary notion of algorithmic randomness that we will consider in this study is Martin-L\"of randomness.
\begin{defn}{\ }
\begin{itemize}
\item[(i)] A \emph{Martin-L\"of test} is a sequence $(\mathcal{U}_i)_{i\in\omega}$ of uniformly effectively open subsets of $\cs$ such that for each $i$,
$
\lambda(\mathcal{U}_i)\leq 2^{-i}
$.

\item[(ii)] $X\in\cs$ \emph{passes} the Martin-L\"of test $(\mathcal{U}_i)_{i\in\omega}$ if $X\notin\bigcap_{i \in \omega}\mathcal{U}_i$.

\item[(iii)] $X\in\cs$ is \emph{Martin-L\"of random}, denoted $X\in\MLR$, if $X$ passes every Martin-L\"of test. 
\end{itemize}
\end{defn}

We will also consider relative versions of Martin-L\"of randomness, obtained by relativizing the above notion of a Martin-L\"of test to some oracle $A\in\cs$; such a class will be written as $\MLR^A$.  For $A=\emptyset^{(n)}$, the resulting notion of randomness is known as $(n+1)$-randomness.  Other randomness notions can be obtained as follows.

\begin{defn} Let $X\in\cs$.
\begin{itemize}
\item[(i)] $X$ is \emph{Schnorr random} if and only if $X$ passes every Martin-L\"of test $(\mathcal{U}_i)_{i\in\omega}$ such that $\lambda(\mathcal{U}_i)$ is computable uniformly in $i\in\omega$.
\item[(i)] $X$ is \emph{Kurtz random} (or \emph{weakly 1-random}) if and only if $X$ is not contained in any $\Pi^0_1$ class of Lebesgue measure~$0$.
\item[(ii)] $X$ is \emph{weakly 2-random} if and only if $X$ is not contained in any $\Pi^0_2$ class of Lebesgue measure~$0$.
\item[(iii)] $X$ is \emph{difference random} if and only if it is Martin-L\"of random and not Turing complete.
\end{itemize}
\end{defn}
Let $\SR$ and $\KR$ denote the collections of Schnorr random and Kurtz random sequences, respectively.

Each of the above notions of tests and randomness can also be formulated for arbitrary computable measures $\mu$ on $\cs$ simply by replacing the Lebesgue measure $\lambda$ in the respective definitions by~$\mu$.  Thus, for instance, for a fixed computable measure $\mu$, a sequence $X$ is $\mu$\nobreakdash-Martin-L\"of random, denoted $X\in\MLR_\mu$, if and only if $X$ is not contained in any $\mu$\nobreakdash-Martin-L\"of test.  Significantly, Martin-L\"of randomness with respect to some computable measure is Turing invariant in the following sense.

\begin{thm}[Levin, Zvonkin~\cite{LevZvo70}; Kautz~\cite{Kau91}]\label{thm-levin-kautz}
For every computable measure $\mu$ and for every non-computable $X\in\MLR_\mu$, there is some $Y\in\MLR$ such that $X\equiv_\T Y$.
\end{thm}

The requirement that $X$ be non-computable is necessary since every computable sequence~$X$ is random with respect to some computable measures on $\cs$, for example the measure $\delta_X$ defined for $\mathcal{A} \subseteq 2^\omega$ via 
\[{\delta_X(\mathcal{A})=\begin{cases}1 & \text{if }X \in \mathcal{A}, \\ 0 & \text{else.} \end{cases}}\]

\section{Negligibility and Non-negligibility}\label{sec-neg-nonneg}

To define the notions of negligibility and non-negligibility, we need to review the definition of left-c.e.\ semi-measures, which were initially introduced by Solomonoff~\cite{Sol64a,Sol64b} and first systematically studied by Levin and Zvonkin \cite{LevZvo70}.

\subsection{Left-c.e.\ semi-measures}
\begin{defn}
A \emph{semi-measure} is a function $P\colon\str\rightarrow[0,1]$ that satisfies
\begin{itemize}
\item[(i)] $P(\varepsilon)\leq 1$,
\item[(ii)] $P(\sigma)\geq P(\sigma0)+P(\sigma1)$ for every $\sigma\in\str$.
\end{itemize}
In addition, $P$ is left-c.e.\ if $P(\sigma)$ is the limit of a computable, non-decreasing sequence of rationals, uniformly in $\sigma\in\str$.
\end{defn}

Functions satisfying conditions (i) and (ii) above are sometimes referred to in the algorithmic randomness literature as \emph{continuous semi-measures} to distinguish them from discrete semi-measures.  As we do
not consider discrete semi-measures in this study, we will not make this distinction below.  

In Section \ref{sec-implementing}, the support of a semi-measure will play an important role.

\begin{defn}
The \emph{support} of a semi-measure $P$, denoted $\supp{P}$ is the collection of sequences 
\[
\{X\in\cs\colon\forall n\; P(X\uh n)>0\}.
\]
\end{defn}

It is not immediately clear how to extend semi-measures to Borel subsets of $\cs$.   Levin and V'yugin~\cite{LevVyu77} proposed the following way of deriving  measures from left-c.e.\ semi-measures.

\begin{defn}
Given a left-c.e.\ semi-measure $P$ and $\sigma\in\str$ we define
\[
\overline{P}(\sigma)=\inf_n\sum_{\sigma\preceq\tau \;\wedge\; |\tau|=n}P(\tau).
\]
\end{defn}
$\overline{P}$ can be extended to a measure  on $2^\omega$, which we will also write as $\overline P$, by letting $\overline{P}(\llb \sigma \rrb) = \overline{P}(\sigma)$ and then applying  Carathéodory's theorem. 
One can show inductively that $\overline P$~is the maximal measure such that $\overline P(\sigma)\leq P(\sigma)$ for every $\sigma\in\str$ (see, for instance, Bienvenu~et~al.~\cite[Proposition 6.5]{BieHolPor14}). As a consequence, $\overline{P}$~is typically not a probability measure.

Inversely, given any computable measure~$\mu$ defined on $\cs$, we can identify it with the left-c.e.\ semi-measure~${\sigma \mapsto \mu(\llbracket \sigma \rrbracket)}$ defined on $\str$; then we have~$\overline\mu=\mu$.

\bigskip

An important property of left-c.e.\ continuous semi-measures is the following.

\begin{thm}[Levin, Zvonkin \cite{LevZvo70}]\label{thm-InduceSemiMeasures}{\ }
\begin{itemize}
\item[(i)] For every Turing functional $\Phi$, the function $\lambda_\Phi$ defined for every $\sigma \in \str$ via
\[
\lambda_\Phi(\sigma)=\lambda(\llb\Phi^{-1}(\sigma)\rrb)=\lambda(\{X\in\cs\colon\Phi^X\succeq\sigma\}),
\]
where $\Phi^X\in\cs\cup\str$, is a left-c.e.\ semi-measure.
\item[(ii)] For every left-c.e.\ semi-measure $P$, there is a Turing functional $\Phi$ such that $P=\lambda_\Phi$. 
\end{itemize}
\end{thm}

Using Theorem \ref{thm-InduceSemiMeasures} one can derive an alternative characterization of $\overline{P}$ for any left-c.e.\ semi-measure $P$.

\begin{prop}\label{prop-barchar}
Let $P$ be a left-c.e.\ semi-measure.  Then
\[
\overline P(\sigma)=\lambda(\{X\in\cs\colon\Phi^X\in\cs\;\wedge\; \Phi^X\succeq\sigma\}),
\]
where $\Phi$ is as in Theorem~\ref{thm-InduceSemiMeasures} (ii).
Moreover, for measurable~$\A \subseteq \cs$, Carathéodory's theorem implies that \[\overline P(\A)=\lambda(\Phi^{-1}(\A)).\]
\end{prop}
\noindent For a proof of the first part of the proposition, see Bienvenu et al.~\cite[Proposition 6.5]{BieHolPor14}.

\begin{thm}[Levin, Zvonkin \cite{LevZvo70}]\label{univ_semi_measure_sdfsdfer}
There is a universal left-c.e.\ semi-measure, that is, a left-c.e.\ semi-measure $M$ such that for every left-c.e.\ semi-measure $P$, there is some constant~$c$ such that
\[
P(\sigma)\leq c\cdot M(\sigma)
\]
for every $\sigma\in\str$.  
\end{thm}

\begin{remark}\label{rmk1}{\ } 
\begin{itemize}
\item[(i)] One way to define a universal semi-measure is via a universal functional.  For instance, for an effective enumeration $(\Phi_e)_{e\in\omega}$ of all Turing functionals, we can define $\Phi\colon\cs\rightarrow\cs$ via $\Phi(1^e0X)=\Phi_e(X)$ for each $e\in\omega$ and $X\in\cs$.  It is not hard to verify that $\lambda_\Phi$ is universal.

\item[(ii)] For every left-c.e.\ semi-measure $P$, there is some $c$ such that
\[
\overline{P}(\sigma)\leq c\cdot \overline{M}(\sigma).
\]
To see this, observe that for the $c$ appearing in Theorem~\ref{univ_semi_measure_sdfsdfer} we have
\[\overline{P}(\sigma) = \inf_n\sum_{\sigma\preceq\tau \;\wedge\; |\tau|=n} P(\tau) \leq \inf_n\sum_{\sigma\preceq\tau \;\wedge\; |\tau|=n} c \cdot M(\tau) = 
c \cdot \overline{M}(\sigma).\]

\item[(iii)] From (ii) and a straightforward argument using open covers of null sets, we can derive the conclusion that for every left-c.e.\ semi-measure $P$, $\overline P$ is absolutely continuous with respect to $\overline M$; that is, if $\overline M(\B)=0$ then $\overline P(\B)=0$ for every measurable set~$\B$.
\end{itemize}
\end{remark}

Using a universal semi-measure we can provide an alternative characterization of $\mu$-Martin-L\"of randomness for each computable measure $\mu$.

\begin{thm}[Levin~\cite{Lev74}; Schnorr, see Chaitin~\cite{Cha75}]\label{thm-levin-schnorr}
Let $\mu$ be a computable measure.  Then $X\in\MLR_\mu$ if and only if there is some $c$ such that $\mu(X\uh n)\geq c\cdot M(X\uh n)$ for every $n$.
\end{thm}

We can now define the notion of negligibility.

\begin{defn}
We say that $\B\subseteq\cs$ is \emph{negligible} if $\overline M(\B)=0$.
\end{defn}

As a consequence of Remark \ref{rmk1}~(iii) we obtain the following corollary.

\begin{corollary}\label{dfgsdafkdfjhsdsdfg}
Let $P$ be a left-c.e.\ semi-measure and $\B\subseteq\cs$ a negligible collection of sequences.  Then $\overline P(\B)=0$.   In particular, $\mu(\B)=0$ for every computable measure $\mu$.
\end{corollary}

Negligibility of a collection can alternatively be characterized by stipulating that {\em no} Turing functional produce an element of that collection with positive probability, as the following proposition shows.

\begin{prop}\label{prop-neg}
Let $(\Phi_i)_{i\in\omega}$ be an effective enumeration of all Turing functionals.  Then a measurable $\B\subseteq\cs$ is negligible if and only if
\[
\lambda\Biggl(\bigcup_{i\in\omega}\Phi_i^{-1}(\B)\Biggr)=0.
\]
\end{prop}

\begin{proof}
($\Rightarrow$:) Suppose that $\lambda\Bigl(\bigcup_{i\in\omega}\Phi_i^{-1}(\B)\Bigr)>0$.  Then there is some $i$ such that $\lambda(\Phi_i^{-1}(\B))>0$.  Setting $P(\sigma)=\lambda(\llb \Phi_i^{-1}(\sigma) \rrb)$ for $\sigma\in\str$,
it follows from Theorem \ref{thm-InduceSemiMeasures}~(i) that $P$ is a left-c.e.\ semi-measure.  
Moreover, we have $\overline P(\B)=\lambda(\Phi_i^{-1}(\B))$
by Proposition~\ref{prop-barchar} and thus $\overline P(\B)>0$.
By Remark~\ref{rmk1}~(iii), $\overline M(\B)>0$, so $\B$ is not negligible.\\
($\Leftarrow$:)  Let $\Phi$ be a Turing functional such that $M=\lambda_\Phi$, which exists by Theorem \ref{thm-InduceSemiMeasures}~(ii).  If $\B$~is not negligible, then we have $0<\overline M(\B)=\lambda(\Phi^{-1}(\B))$ by Proposition~\ref{prop-barchar}, and hence \[\lambda\Bigl(\bigcup_{i\in\omega}\Phi_i^{-1}(\B)\Bigr)>0.\qedhere \]
\end{proof}

Intuitively, a  collection of sequences is negligible if none of its elements can be obtained with positive probability by any probabilistic algorithm. Indeed, we can see a probabilistic algorithm as consisting of two steps: First we generate infinitely many random bits, then we feed them to some Turing functional to produce the desired output.   
More formally, we can think of a probabilistic algorithm as given by applying a Turing functional $\Phi$ to some random sequence.  In this case, we can probabilistically compute an element of some fixed collection $\B$ with positive probability if there are positive measure many sequences $X$ such that $\Phi(X)\in\B$.  Proposition \ref{prop-neg} tells us that the existence of such a probabilistic algorithm to compute elements of $\B$ with positive probability is equivalent to the non-negligibility of $\B$.

We conclude this subsection with a brief discussion of the atoms of a semi-measure.  

\begin{defn}
Let $P$ be a semi-measure.  $X\in\cs$ is an \emph{atom} of $P$ if there is some $\delta>0$ such that $P(X\uh n)>\delta$ for all $n$.
\end{defn}

\begin{lem}
Let $P$ be a semi-measure. $X\in\cs$ is an atom of $P$ if and only if $\overline{P}(\{X\})>0$.
\end{lem}

\begin{proof}
($\Rightarrow$:)  If there is some $\delta>0$ such that $P(X\uh n)>\delta$ for all $n$, then for each $n$ and each~${m\geq n}$,
\[
\sum_{X\uh n\preceq\tau \;\wedge\; |\tau|=m}P(\tau)\geq P(X\uh m)>\delta.
\]
It follows from the definition of $\overline{P}$ that $\overline{P}(X\uh n)>\delta$ for all $n$.

\noindent ($\Leftarrow$:) ${\overline{P}(\{X\})>0}$ implies that there is an $\delta>0$ such that  $\overline{P}(X\uh n)>\delta$ for all $n$. Then, for all~$n$,
\[P(X\uh n)\geq \overline{P}(X\uh n) >    \delta.\qedhere\]
\end{proof}

\begin{prop}[Bienvenu et al.~\cite{BieHolPor14}]\label{prop_atoms_are_comp} 
Let $P$ be a left-c.e.\ semi-measure. If $X$ is an atom of~$P$, then $X$ is computable.
\end{prop}

\subsection{Examples of negligible and non-negligible collections}

We now provide a number of examples of negligible and non-negligible collections of sequences, where the first set of examples is given by a classical theorem of Sacks.

\begin{thm}[Sacks~\cite{Sac63}]
For $X\in\cs$, $\lambda(\{Y\in\cs\colon Y\geq_\T X\})>0$ if and only if $X$ is computable.  That is, $\{X\}$~is non-negligible if and only if $X$ is computable.
\end{thm}

Arbitrary subsets of $\cs$ of positive Lebesgue measure are further trivial examples of non-negligible collections.  Thus, each of the notions of randomness defined above in Subsection~\ref{subsec-notions} forms a non-negligible collection.

We can find more interesting examples by considering naturally occurring collections of Turing degrees.  We briefly review some of these collections.  First, a sequence has \emph{PA degree} if it computes a consistent completion of Peano arithmetic.  A sequence $X\in\cs$ is \emph{high} (or has \emph{high Turing degree}) if and only if $\{X\in\cs\colon X''\geq_\T\emptyset'\}$.  A sequence $X\in\cs$ is \emph{1-generic} if for every c.e.\ $S\subseteq\str$, there is some $\sigma\prec X$ such that either $\sigma\in S$ or for all $\tau\succeq\sigma$, $\tau\notin S$.  Similarly, $X\in\cs$ is \emph{2-generic} if for every $\emptyset'$-c.e.\ $S\subseteq\str$, there is some $\sigma\prec X$ such that either $\sigma\in S$ or for all $\tau\succeq\sigma$, $\tau\notin S$.  Next, $X\in\cs$ has \emph{hyperimmune-free degree} if and only if every $X$-computable function is dominated by some computable function.  Accordingly, $X$ has \emph{hyperimmune degree} if and only if $X$ computes a function that is not dominated by any computable function.
$X\in\cs$ is of \emph{DNC degree} if and only if there is some $f\leq_\T X$ such that $f(e)\neq \phi_e(e)$ for all $e\in\omega$.  Lastly, $X$ is \emph{generalized low} (or is in GL$_1$) if and only if $X'\equiv_\T X\oplus \emptyset'$.

To establish the negligibility or non-negligibility of the various collections given above, we will use the following heuristic principles, which are justified by Proposition \ref{prop-neg}.  
\begin{itemize}
\item[($P_1$)] \emph{If every sufficiently random sequence computes an element of some measurable $\B\subseteq\cs$, then $\B$ is non-negligible.}
\item[($P_2$)] \emph{If no sufficiently random sequence computes an element of some measurable $\B\subseteq\cs$, then $\B$ is negligible.}

\end{itemize}

\begin{prop} \label{prop-non-negligible}
The following collections are non-negligible:
\begin{enumerate}
\item[(i)] the collection of sequences of DNC degree, 
\item[(ii)] the collection of 1-generic sequences, 
\item[(iii)] the collection of sequences of hyperimmune degree, and
\item[(iv)] the collection of generalized low sequences.
\end{enumerate}
\end{prop}

\begin{proof}
To show that each of the above collections is non-negligible, we apply ($P_1$) by identifying a notion of randomness such that every sequence that is random in the respective sense computes an element of the given collection.  
For~(i), Ku\v cera \cite{Kuc85} proved that every Martin-L\"of random sequence is of DNC degree.  
For~(ii), Kautz~\cite{Kau91} established that every 2-random sequence computes a 1-generic.  Since every 1-generic sequence has hyperimmune degree, it further follows that every 2-random sequence computes a sequence of hyperimmune degree, yielding (iii).  Lastly, for~(iv), Kautz \cite{Kau91} also proved that every 2-random sequence is generalized low.
\end{proof}

\begin{prop} \label{prop-negligible}
The following collections are negligible:
\begin{enumerate}
\item[(i)] the collection of sequences of PA degree,
\item[(ii)] the collection of sequences of high degree, 
\item[(iii)] the collection of 2-generic sequences, and
\item[(iv)] the collection of non-computable sequences of hyperimmune-free degree.
\end{enumerate}
\end{prop}

\begin{proof} To show that each of the above collections is negligible, we apply ($P_2$) by identifying 
a notion of randomness such that no sequence that is random in the respective sense computes an element of the given collection.  
For (i), Franklin and Ng~\cite{FranklinNg} extended work of Stephan~\cite{franktechreport} to show that no difference random sequence computes a completion of PA.
For (ii), Kautz~\cite{Kau91} established that no 3-random has high degree. As the high degrees are closed upwards under Turing reducibility, this implies that no 3-random computes a sequence of high degree.  For (iii), Nies, Stephan, and Terwijn~\cite{NieSteTer05} proved that every 2-random sequence forms a minimal pair in the Turing degrees with every 2-generic, and so no 2-random computes a 2-generic.  Lastly, for (iv), Lewis, Day, and Barmpalias~\cite[Theorem 5.1]{BarDayLew14} showed that for every $2$\nobreakdash-random sequence $X$, every non-computable~$Y\leq_T X$ computes a 1-generic sequence and therefore in particular a sequence of hyperimmune degree. So if any $2$-random could compute a non-computable sequence of hyperimmune-free degree, then this sequence could in turn compute a sequence of hyperimmune degree, contradicting the fact that hyperimmune-freeness is closed downwards under Turing reducibility.
\end{proof}

\section{The Levin-V'yugin Degrees}\label{sec-lv-degrees}

Using the notion of negligibility, we can define a degree structure whose elements are given by Turing invariant subsets of~$\cs$.  Recall that $\A\subseteq\cs$ is Turing invariant if $X\in\A$ and $Y\equiv_\T  X$ imply~$Y\in \A$.  Let $\mathcal{I}$ denote the set of measurable Turing invariant subsets of $\cs$.  In what follows, all Turing invariant collections of sets that we consider are Borel and thus measurable.   One can routinely verify that $(\mathcal{I}, \cap,\cup,^c)$ is a Boolean algebra.  

We now define a reducibility~$\leq_\LV$ on~$\mathcal{I}$.

\begin{samepage}\begin{defn}
	Let $\A,\B\in \mathcal{I}$.
	\begin{itemize}
		\item[(i)] $\A\leq_\LV \B$ if and only if $\A\setminus\B$ is negligible. 
		\item[(ii)] $\A\equiv_\LV \B$ if and only if $\A\leq_\LV \B$ and $\B\leq_\LV \A$.
	\end{itemize}
\end{defn}\end{samepage}

Given $\A,\B\in \mathcal{I}$,  $\A\leq_\LV \B$ says that, for any probabilistic algorithm, the probability that it produces an element of $\A$ that is not in $\B$ is~$0$. The stronger statement $\A<_\LV \B$ says in~addition that there is some probabilistic algorithm such that the probability that it produces an element of~$\B$ that is not in~$\A$ is strictly positive. In this sense, the larger a collection of sets is with regards to the given order, the easier it is to probabilistically produce an element of it.

It is well-known that a Boolean algebra modulo an equivalence relation is still a Boolean algebra.  Thus, $\mathcal{D}_\LV=\mathcal{I}/{\equiv_\LV}$ is a Boolean algebra, which we refer to as the \emph{Levin-V'yugin algebra}.  In fact, $\mathcal{D}_\LV$ is a measure algebra, since it is a Boolean algebra of measurable sets modulo $\overline M$-null sets.  Individual elements of $\D_\LV$ will be referred to as $\LV$\emph{-degrees}.  We will write $\LV$-degrees as $\mathbf{a},\mathbf{b},\dotsc$ and so on.  For $\mathcal{A}\in\mathcal{I}$, $\mathbf{deg_\mathbf{LV}(\mathcal{A})}$ denotes the $\LV$-degree of $\mathcal{A}$.  Given $\LV$-degrees $\mathbf{a}$ and $\mathbf{b}$ and any $\mathcal{A}\in\mathbf{a}$ and $\mathcal{B}\in\mathbf{b}$, we define
\begin{itemize}
\item[] $\mathbf{a}\wedge \mathbf{b}:=\mathbf{deg_\mathbf{LV}(\mathcal{A}\cap\mathcal{B})}$, 
\item[] $\mathbf{a}\vee \mathbf{b}:=\mathbf{deg_\mathbf{LV}(\mathcal{A}\cup\mathcal{B})}$, and
\item[] $\mathbf{a}^c:=\mathbf{deg_\mathbf{LV}}(\cs\setminus\mathcal{A})$.
\end{itemize}
It is straightforward to verify that these are well-defined. 
With slight abuse of notation,  we let~$\leq_\LV$ denote the order on $\D_\LV$ that is induced by the order $\leq_\LV$ on $\mathcal{I}$ modulo the equivalence relation~$\equiv_\LV$; that is, we write 
write $\mathbf{a}\leq_\LV\mathbf{b}$, for two $\LV$-degrees $\mathbf{a}$ and $\mathbf{b}$, if there exist~$\mathcal{A}\in\mathbf{a}$ and~$\mathcal{B}\in\mathbf{b}$ such that~$\mathcal{A}\leq_\LV\mathcal{B}$.
Then the following is immediate.
\begin{prop}\
\begin{itemize}
\item[(i)] The bottom element $\mathbf{0}$ of $\mathcal{D}_\LV$ consists of the Turing invariant negligible subsets of $\cs$.
\item[(ii)] The top element $\mathbf{1}$ of $\D_\LV$ consists of all Turing invariant $\A\subseteq\cs$ such that $\cs\setminus\A$ is negligible.
\end{itemize}
\end{prop}

\subsection{Elementary properties of the $\LV$-degrees}

Recall that $A$ is an atom 
 of a Boolean algebra~$\B$ if there are no elements $A_0$, $A_1\in\B\setminus\{0\}$ such that $A=A_0\vee A_1$ and $A_0\wedge A_1=0$.  To avoid confusion with the atoms of a semi-measure, we will hereafter refer to atoms of $\D_\LV$ as \emph{$\D_\LV$-atoms}.
As reported by V'yugin~\cite{Vyu82}  in results attributed to Levin, two $\D_\LV$-atoms are readily identifiable:  the $\LV$-degree of the computable sequences, denoted~$\mathbf{c}$, and the $\LV$-degree of the Martin-L\"of random sequences, denoted $\mathbf{r}$.  We provide the proofs of these results here.

For $\mathcal{A}\subseteq\cs$, let $\dspec_\T (\mathcal{A})=\{\deg_\T (X)\colon X\in\mathcal{A}\}$ be the \emph{Turing degree spectrum of $\mathcal{A}$}. The following basic fact will be useful.

\begin{samepage}
\begin{lem}\label{fact1}
Given $\mathbf{a_0},\mathbf{a_1}\in\mathcal{D}_\LV$ such that $\mathbf{a_0}\wedge\mathbf{a_1}=\mathbf{0}$, there are $\A_0,\A_1\in\mathcal{I}$ such that 
\begin{itemize}
\item[(i)] $\dspec_\T (\A_0)\cap\dspec_\T (\A_1)=\emptyset$ and
\item[(ii)] $\mathbf{deg_\mathbf{LV}(}\A_0\mathbf{)}=\mathbf{a_0}$ and $\mathbf{deg_\mathbf{LV}(}\A_1\mathbf{)}=\mathbf{a_1}$.
\end{itemize}
Furthermore, for any given $\A\in\mathcal{I}$ satisfying $\mathbf{deg_\mathbf{LV}(}\A\mathbf{)}=\mathbf{a_0\vee a_1}$, we can w.l.o.g.\ assume that
\begin{itemize}
\item[(iii)] $\A_i\subseteq \A$ for $i=0,1$.
\end{itemize}
\end{lem}
\end{samepage}

\begin{proof}
The statement $\mathbf{a_0}\wedge\mathbf{a_1}=\mathbf{0}$ says that if we pick \textit{any} element $\B_0 \in \mathcal{I}$ of the equivalence class $\mathbf{a_0}$ and \textit{any} element $\B_1 \in \mathcal{I}$ of the equivalence class $\mathbf{a_1}$, then $\B_0 \cap \B_1$ is negligible. Then  $\A_0 := \B_0 \setminus \B_1 \equiv_\LV \B_0$ is in the equivalence class $\mathbf{a_0}$, $\A_1 := \B_1 \setminus \B_0 \equiv_\LV \B_1$ is in $\mathbf{a_1}$, and since $\B_0$ and $\B_1$ are closed under Turing equivalence we also have ${\dspec_\T (\A_0)\cap\dspec_\T (\A_1)=\emptyset}$.

To verify (iii), suppose that  $\mathbf{deg_\mathbf{LV}(}\A\mathbf{)}=\mathbf{a_0\vee a_1}$ for some $\A\in\mathcal{I}$ and
let $\A_0^\prime$ and $\A_1^\prime$ satisfy conditions (i) and (ii) above.  Then $\mathbf{deg_\mathbf{LV}(}\A\mathbf{)}=\mathbf{deg_\mathbf{LV}(}\A_0^\prime\cup\A_1^\prime\mathbf{)}$, which implies that $\A\Delta(\A_0^\prime\cup\A_1^\prime)$ is negligible. As $\A_0^\prime$ and $\A_1^\prime$ are disjoint, this implies that  $\A_i^\prime\setminus\A$ is negligible for $i=0,1$. For~${i=0,1}$, setting $\A_i= \A_i^\prime\cap\A$, we have
\[
\A_i^\prime= (\A_i^\prime \cap \A) \cup (\A_i^\prime\setminus \A) = \A_i \cup (\A_i^\prime\setminus \A).
\]
Thus, $\A_i^\prime$ and $\A_i$ differ only by a negligible set for $i=0,1$, and thus $\A_0$ and $\A_1$ satisfy~(ii).  Moreover, since $\A_i\subseteq\A_i^\prime$ for $i=0,1$, $\A_0$ and $\A_1$ also satisfy~(i). Thus, (iii)~holds.
\end{proof}

\begin{prop}\label{prop-comp-atom}
$\mathbf{c}$ is a $\D_\LV$-atom.
\end{prop}

\begin{proof}
Suppose that $\mathbf{c}$ is not a $\D_\LV$-atom.  Then there are $\LV$-degrees $\mathbf{a_0},\mathbf{a_1}>\mathbf{0}$ such that $\mathbf{a_0} \wedge\mathbf{a_1}=\mathbf{0}$ and $\mathbf{a_0}\vee\mathbf{a_1}=\mathbf{c}$.  Then,
if we choose $\A$ in condition~(iii) of Lemma~\ref{fact1} as the collection of all computable sequences, there are $\A_0,\A_1\in\mathcal{I}$ satisfying all three conditions of that lemma. But clearly, conditions~(i) and~(iii) are in contradiction with each other in this case.
\end{proof}

\begin{thm}\label{thm:rand-atom}
$\mathbf{r}$ is a $\D_\LV$-atom.
\end{thm}

To prove Theorem \ref{thm:rand-atom}, we will need to draw upon several classical results from measure theory, as well as several auxiliary lemmata.  Here we follow V'yugin's general proof strategy while filling in more details, especially in isolating and proving Lemma \ref{lem:rand-atom-1} below.

As noted in Remark \ref{rmk1}~(iii), for any left-c.e.\ semi-measure $P$, $\overline{P}$ is absolutely continuous with respect to $\overline{M}$.  It follows by the Radon-Nikodym Theorem that there is a measurable  function~$\frac{d\overline{P}}{d\overline{M}}$ such that, for all measurable $\X \subseteq 2^\omega$,
\[
\overline{P}(\X)=\int_\X{\frac{d\overline{P}}{d\overline{M}}}(X)d\overline{M}(X). 
\]
The Radon-Nikodym Theorem further guarantees that for any measurable  $f\!\colon\cs\rightarrow\mathbb{R}$  such that for all measurable $\X \subseteq 2^\omega$ the property
\[
\overline{P}(\X)=\int_\X f(X)d\overline{M}(X)
\]
holds, we have $f(X)=\dfrac{d\overline{P}}{d\overline{M}}(X)$ for $\overline{M}$-almost every $X\in\cs$.  

\begin{lem}\label{lem:rand-atom-1}
$\dfrac{d\overline{P}}{d\overline{M}}(X)=\lim_{n\rightarrow\infty}\dfrac{\overline{P}(X\uh n)}{\overline{M}(X\uh n)}$ for $\overline{M}$-almost every $X\in\cs$.
\end{lem}

\begin{proof}

First, recall that for a measure $\mu$ on $\cs$, a $\mu$-martingale is a function $d\colon\str\rightarrow \mathbb{R}^{\geq 0}$ such that 
\[
\mu(\sigma)d(\sigma)=\mu(\sigma0)d(\sigma0)+\mu(\sigma1)d(\sigma1)
\]
for every $\sigma\in\str$.\footnote{See, for instance, Nies~\cite[Chapter 7]{Nie09} or Downey and Hirschfeldt~\cite[Section 6.3]{DowHir10} for a discussion of the role of martingales in the theory of algorithmic randomness.}

Now, observe that $\dfrac{\overline{P}}{\,\overline{M}\,}$ is an $\overline{M}$-martingale. Indeed,  for every $\sigma\in\str$,
\[
\overline{M}(\sigma)\dfrac{\overline{P}(\sigma)}{\overline{M}(\sigma)}=\overline{P}(\sigma)=\overline{P}(\sigma0)+\overline{P}(\sigma1)=
\overline{M}(\sigma0)\dfrac{\overline{P}(\sigma0)}{\overline{M}(\sigma0)}+\overline{M}(\sigma1)\dfrac{\overline{P}(\sigma1)}{\overline{M}(\sigma1)}.
\]
   Thus $\lim_{n\rightarrow\infty}\dfrac{\overline{P}(X\uh n)}{\overline{M}(X\uh n)}$ exists for $\overline{M}$-almost every $X\in\cs$ by the martingale convergence theorem.\footnote{It is well known that every martingale in the sense of algorithmic randomness (as given above) is a martingale in the classical sense, and thus the classical martingale convergence theorem is applicable.  See Downey and Hirschfeldt~\cite[Theorem~7.1.3]{DowHir10} for a proof of an effective version of the martingale convergence theorem.}  Thus, by the Radon-Nikodym theorem, we just need to show that 
\begin{equation*}\tag{$\dagger$}
\overline{P}(\A)=\int_\A \lim_{n\rightarrow\infty}\dfrac{\overline{P}(X\uh n)}{\overline{M}(X\uh n)} d\overline{M}(X)
\end{equation*}
for every clopen $\A\subseteq\cs$ (which can then can be extended to every measurable  $\mathcal{A}\subseteq\cs$).
Since there is some $c$ such that $\overline{P}(\sigma)\leq c\cdot\overline{M}(\sigma)$ for every $\sigma\in\str$, we have for every $n$ that
\[
\dfrac{\overline{P}(X\uh n)}{\overline{M}(X\uh n)}\leq c,
\]and hence by the dominated convergence theorem,
\begin{equation*}\tag{$\ddagger$}
\lim_{n\rightarrow\infty}\int_\A \dfrac{\overline{P}(X\uh n)}{\overline{M}(X\uh n)} d\overline{M}(X)=
\int_\A \lim_{n\rightarrow\infty}\dfrac{\overline{P}(X\uh n)}{\overline{M}(X\uh n)} d\overline{M}(X).
\end{equation*}
Using ($\dagger$), it now suffices to show that $\overline P(\mathcal{A})$ is equal to the left-hand side of ($\ddagger$).  For each sufficiently large $N$, let $\A=\bigcup_{i=1}^k\llbracket\sigma_i\rrbracket$ for distinct $\sigma_1,\dotsc,\sigma_k\in2^N$.  Then

\allowdisplaybreaks

\begin{align}
\lim_{n\rightarrow\infty}\int_\A \dfrac{\overline{P}(X\uh n)}{\overline{M}(X\uh n)} d\overline{M}(X)&=\int_\A \dfrac{\overline{P}(X\uh N)}{\overline{M}(X\uh N)} d\overline{M}(X)\\
&=\sum_{i=1}^k\int_{\llbracket\sigma_i\rrbracket}\dfrac{\overline{P}(X\uh N)}{\overline{M}(X\uh N)} d\overline{M}(X)\\
&=\sum_{i=1}^k\dfrac{\overline{P}(\llbracket\sigma_i\rrbracket)}{\overline{M}(\llbracket\sigma_i\rrbracket)}\overline{M}(\llbracket\sigma_i\rrbracket)\\
&=\sum_{i=1}^k\overline{P}(\llbracket\sigma_i\rrbracket)=\overline{P}(\A).\qedhere
\end{align}
\end{proof}

\begin{lem}[V'yugin \cite{Vyu82}]\label{lem:rand-atom-2}
Let $P$ be a left-c.e.\ semi-measure and suppose that for $\B\subseteq\cs$, we have $\overline{M}(\B_0)=0$, where
\[
\B_0=\Biggl\{X\in\B\colon \frac{d\overline{P}}{d\overline{M}}(X)=0\Biggr\}.
\]
Then $\overline{P}(\B)=0$ implies that $\overline{M}(\B)=0$.
\end{lem}

\begin{proof}
By the hypothesis,
\[
0=\overline P(\B\setminus\B_0)=\int_{\B\setminus\B_0}\frac{d\overline{P}}{d\overline{M}}(X)d\overline{M}(X).
\]
Since $\dfrac{d\overline{P}}{d\overline{M}}(X)\neq 0$ for every $X\in\B\setminus\B_0$, it follows that $\overline{M}(\B\setminus\B_0)=0$.  Thus, $\overline M(\B)=0$.
\end{proof}

\begin{lem}[V'yugin \cite{Vyu82}]\label{lem:rand-atom-3}
Let $\mu$ be a computable measure, and let $\B\subseteq\MLR_\mu$ be such that~$\mu(\B)=0$.  Then $\B$ is negligible.
\end{lem}

\begin{proof}
Since $\B\subseteq\MLR_\mu$, by Theorem \ref{thm-levin-schnorr}, for every $X\in \B$, there is some $c$ such that 
\[
\mu(X\uh n)\geq c\cdot M(X\uh n)
\]
for every $n$.  It follows that for all $n$,
\[
\frac{\mu(X\uh n)}{\overline{M}(X\uh n)}\geq\frac{\mu(X\uh n)}{M(X\uh n)}\geq c.
\]
By Lemma \ref{lem:rand-atom-1}, $\dfrac{d\mu}{d\overline M}(X)\neq 0$ for $\overline{M}$-almost every $X\in\B$, and so by Lemma \ref{lem:rand-atom-2} and the fact that~${\mu(\B)=0}$, it follows that $\B$ is negligible.
\end{proof}

Lastly, we need one further classical result.  Recall that $\A\subseteq\cs$ is a tailset if for all $\sigma\in\str$ and all~$Y\in\cs$ with $\sigma Y\in\A$ we also have that $\tau Y\in\A$ for every $\tau \in 2^{|\sigma|}$. That is, for a tailset~$\A$, modifying a finite initial segment of an infinite binary sequence has no bearing on whether that sequence is an element of $\A$ or not. The following result will only be used in the context of Cantor space; for a proof specific to that setting see Downey and Hirschfeldt~\cite[Theorem~1.2.4]{DowHir10}.
\begin{thm}[Kolmogorov's 0-1 Law]\label{thm:tailset}
If $\A\subseteq\cs$ is a measurable tailset, then $\lambda(\A)=0$ or~${\lambda(\A)=1}$. 
\end{thm}
 We can now prove Theorem~\ref{thm:rand-atom}.
\begin{proof}[Proof of Theorem \ref{thm:rand-atom}]
Suppose that $\mathbf{r}=\mathbf{a_0}\vee\mathbf{a_1}$ and $\mathbf{a_0}\wedge\mathbf{a_1}=\mathbf{0}$ for some $\mathbf{a_0},\mathbf{a_1}>\mathbf{0}$.  Let ${\A_0,\A_1\in\mathcal{I}}$ be collections of sequences as given by Lemma~\ref{fact1} where ${\mathbf{deg_\mathbf{LV}(}\A_i\mathbf{)}=\mathbf{a}_i}$ and $\A_i\subseteq(\MLR)^{\equiv_\T }$ for~${i=0,1}$.    Note that for $i=0,1$, for each $X\in\A_i$ there is some $Y\in\MLR\cap \A_i$ such that $X\equiv_\T  Y$. Let us consider the subcollections of sequences $\A_i^*=\MLR\cap \A_i$ for $i=0,1$.  Since each~$\A_i$ is non-negligible, it follows that
\[
\lambda\Biggl(\bigcup_e\Phi^{-1}_e(\A_i)\Biggr)>0
\]
for $i=0,1$.  Since each $X\in\A_i$ is Turing equivalent to some $Y\in\A^*_i$, it follows for $i=0,1$ that 
\[
\bigcup_e\Phi^{-1}_e(\A_i)=\bigcup_e\Phi^{-1}_e(\A^*_i)
\]
and hence
\[
\lambda\Biggl(\bigcup_e\Phi^{-1}_e(\A^*_i)\Biggr)>0.
\]
 Then Proposition~\ref{prop-neg} and Lemma \ref{lem:rand-atom-3} imply that $\lambda(\A_i^*)>0$ for $i=0,1$.  But each ~$\A^*_i$~is a measurable tailset, so by Theorem \ref{thm:tailset} it follows that  
 $\lambda(\A^*_i)=1$ for $i=0,1$, which is impossible as $\A^*_0$ and $\A^*_1$ 
 are disjoint.
\end{proof}

\subsection{Additional results about the $\LV$-degrees}
It is reasonable to ask whether the degree $\mathbf{r}\vee\mathbf{c}$ is 
the top degree in $\D_\LV$.  V'yugin gave a negative answer to this question by proving that the complement of $\mathbf{r}\vee\mathbf{c}$ in $\D_\LV$ is non-negligible.
We will give the details of his proof in Section~\ref{sec-implementing}, where we will provide the first instance of the technique of building semi-measures that we mentioned in the introduction.  
However, in this subsection, we provide a simpler proof of this result, and a number of new results about $\D_\LV$.

 Given $\mathbf{a}\in\mathcal{D}_\LV$ and $\A\subseteq\cs$ such that $\mathbf{deg_\mathbf{LV}(}(\A)^{\equiv_\T }\mathbf{)}=\mathbf{a}$, we say that $\A$ \emph{generates} $\mathbf{a}$ or that $\mathbf{a}$ is the $\LV$-degree \emph{generated by} $\A$.  We will use the following lemma repeatedly.

\begin{samepage}
\begin{lem}\label{lem-negfacts} Let $\A,\B\subseteq \cs$ be measurable sets.
\begin{itemize} 
\item[(i)]  If $\A\setminus\B$ is negligible, then $(\A)^{\equiv_\T }\setminus(\B)^{\equiv_\T }$ is also negligible.  In particular, we have~${(\A)^{\equiv_\T }\leq_\LV(\B)^{\equiv_\T }}$.
\item[(ii)] If $\A\subseteq \B$, then $(\A)^{\equiv_\T }\leq_\LV(\B)^{\equiv_\T }$.
\end{itemize}
\end{lem}
\end{samepage}

\begin{proof}
(i) First observe that $(\A)^{\equiv_\T }\setminus(\B)^{\equiv_\T }\subseteq (\A\setminus\B)^{\equiv_\T }$.  Indeed, given $X\in(\A)^{\equiv_\T }\setminus(\B)^{\equiv_\T }$, there is some $Y\equiv_\T X$ such that $Y\in\A$ and for all $Z\in\B$, we have $Z\not\equiv_\T X$.  It follows that~${Y\notin \B}$, and hence $X\in  (\A\setminus\B)^{\equiv_\T }$.

Now suppose that $(\A)^{\equiv_\T }\setminus(\B)^{\equiv_\T }$ is non-negligible.  By the above observation, $(\A\setminus\B)^{\equiv_\T }$ is also non-negligible.  For $i,j\in\omega$ define $\S_{i,j}=\{X\in\cs\colon (\exists Y\in \A\setminus \B)\; (\Phi_i(Y)=X \;\wedge\;\Phi_j(X)=Y)\}$.  Then we have
\[
(\A\setminus\B)^{\equiv_\T }=\bigcup_{(i,j)\in\omega^2}\S_{i,j}.
\]
Since $(\A\setminus\B)^{\equiv_\T }$ is non-negligible, there is some pair $(i,j)\in\omega^2$ such that $\S_{i,j}$ is non-negligible. Then by Proposition \ref{prop-neg}, there is some Turing functional $\Psi$ such that $\lambda(\Psi^{-1}(\S_{i,j}))>0$.  By definition of $\S_{i,j}$, if $\Psi(Z)\in\S_{i,j}$, then $\Phi_j(\Psi(Z))\in\A\setminus\B$. 
Thus $\Psi^{-1}(\S_{i,j})\subseteq (\Phi_j\circ\Psi)^{-1}(\A\setminus\B)$, and so $\lambda((\Phi_j\circ\Psi)^{-1}(\A\setminus\B))>0$.  Thus by Proposition \ref{prop-neg}, $\A\setminus\B$ is not negligible.

\smallskip

\noindent (ii) If $\A\subseteq \B$, then $\A\setminus\B=\emptyset$ is trivially negligible.  Thus by~(i), $(\A)^{\equiv_\T }\leq_\LV(\B)^{\equiv_\T }$.
\end{proof}

It is natural to ask how the $\LV$-degree of the Martin-L\"of random Turing degrees compares to the $\LV$-degrees associated to other notions of algorithmic randomness.   First we show
that the $\LV$-degree of the Schnorr random Turing degrees is also $\mathbf{r}$.

\begin{thm}\label{thm-srvd}
$\mathbf{deg_\mathbf{LV}(}(\SR)^{\equiv_\T }\mathbf{)}=\mathbf{r}$.
\end{thm}

\begin{proof}
($\geq_\LV$:) $\MLR\subseteq\SR$, and thus 
by Lemma \ref{lem-negfacts} (ii), $(\MLR)^{\equiv_\T }\leq_\LV(\SR)^{\equiv_\T }$.  \\
($\leq_\LV$:)  We show that $\mathrm{SR}\setminus\mathrm{MLR}$ is negligible, which by  Lemma \ref{lem-negfacts} (i) implies $(\SR)^{\equiv_\T }\leq_\LV(\MLR)^{\equiv_\T }$.  
As shown by Nies, Stephan, and Terwijn \cite{NieSteTer05}, every $X\in\SR\setminus\MLR$ has high degree.  But by Proposition \ref{prop-negligible}, the collection of sequences of high degree is negligible.  
\end{proof}

\begin{corollary}\label{cor-srvd}
Let $\mathrm{R}$ be any notion of algorithmic randomness such that $\MLR\subseteq\mathrm{R}\subseteq\SR$.  Then \[\mathbf{deg_\mathbf{LV}(}(\mathrm{R})^{\equiv_\T }\mathbf{)}=\mathbf{r}.\]
\end{corollary}

\begin{proof}  
By Lemma  \ref{lem-negfacts} (ii) and Theorem \ref{thm-srvd}, we have 
\[
\mathbf{r}=\mathbf{deg_\mathbf{LV}(}(\MLR)^{\equiv_\T }\mathbf{)}\leq_\LV	\mathbf{deg_\mathbf{LV}(}(\mathrm{R})^{\equiv_\T }\mathbf{)}\leq_\LV\mathbf{deg_\mathbf{LV}(}(\SR)^{\equiv_\T }\mathbf{)}=\mathbf{r}.\qedhere
\]
\end{proof}

Thus, notions of randomness such as computable randomness, Kolmogorov-Loveland randomness, and the non-monotonic randomness notions studied in Bienvenu et al.~\cite{BieHolKra12} all are of $\LV$-degree $\mathbf{r}$.  Similar results hold for notions of randomness stronger than Martin-L\"of randomness, as the following result shows.

\begin{thm}\label{thm-relrvd}
For every $Z\in\cs$, $\mathbf{deg_\mathbf{LV}(}(\MLR^Z)^{\equiv_\T }\mathbf{)}=\mathbf{deg_\mathbf{LV}(}(\MLR)^{\equiv_\T }\mathbf{)}$.
\end{thm}

\begin{proof}
($\geq_\LV$:) $\MLR^Z\subseteq\MLR$, and so by Lemma \ref{lem-negfacts} (ii), $(\MLR^Z)^{\equiv_\T }\leq_\LV(\MLR)^{\equiv_\T }$.  \\
($\leq_\LV$:)  We show that $\mathrm{MLR}\setminus\mathrm{MLR^Z}$ is negligible and apply Lemma \ref{lem-negfacts} (i).
Given any ${X\in\MLR\setminus\MLR^Z}$, by the \emph{XYZ} Theorem of Miller and Yu \cite{MilYu08}, if $X\leq_\T  Y\in\MLR^Z$, then~$X\in\MLR^Z$.  Thus no $Y\in\MLR^Z$ computes any $X\in\MLR\setminus\MLR^Z$.  That is, no sufficiently random sequence computes an element of $\MLR\setminus\MLR^Z$, and so by our heuristic~($P_2$), this latter collection is negligible.
\end{proof}

An immediate consequence of Theorem \ref{thm-relrvd} is that for each $n\in\omega$, the $\LV$-degree of the collection of $n$-random sequences is $\mathbf{r}$.  Another consequence is the following, the proof of which is analogous to that of Corollary \ref{cor-srvd}.

\begin{corollary}\label{cor-strong}
Let $\mathrm{R}$ be any notion of algorithmic randomness such that $\MLR^{\emptyset'}\subseteq\mathrm{R}\subseteq\MLR$.  Then \[\mathbf{deg_\mathbf{LV}(}(\mathrm{R})^{\equiv_\T }\mathbf{)}=\mathbf{r}.\]
\end{corollary}

It follows that notions of randomness such as difference randomness, Demuth randomness, and weak 2-randomness all generate the $\LV$-degree $\mathbf{r}$.  

We now show that $\mathbf{r}\vee\mathbf{c}$ is not the top $\LV$-degree by exhibiting an $\LV$-degree that is incomparable with it.  Let $\mathbf{g}$ be the $\LV$-degree generated by the collection of 1-generic sequences.
By Proposition~\ref{prop-non-negligible} this collection is non-negligible. 

\begin{prop}\label{prop-r-g}
{\ }
\begin{itemize}
\item[(i)] $\mathbf{r}\wedge\mathbf{g}=\mathbf{0}$, and hence $\mathbf{r},\mathbf{g}<_\LV\mathbf{r}\vee\mathbf{g}$.
\item[(ii)] $(\mathbf{r}\vee\mathbf{c})\wedge\mathbf{g}=\mathbf{0}$.
\item[(iii)] $\mathbf{r}\wedge(\mathbf{g}\vee\mathbf{c})=\mathbf{0}$.
\end{itemize}
\end{prop}

\begin{proof}
(i) As shown by Demuth and Ku\v cera \cite{DemKuc87}, no 1-generic can compute a Martin-L\"of random sequence.  Thus the set of Turing degrees containing a Martin-L\"of random sequence is disjoint from the set of Turing degrees containing a 1-generic sequence, from which the first part of~(i) follows.  The second part of~(i) immediately follows from the first part.
Statements~(ii) and~(iii) follow from~(i) and the fact that the collection of computable sequences is disjoint from the collection of 1-generic sequences and from the collection of Martin-L\"of random sequences. 
\end{proof}

\begin{corollary}
Neither $\mathbf{r}\vee\mathbf{c}$ nor $\mathbf{g}\vee\mathbf{c}$ equals the top $\LV$-degree $\mathbf{1}$.
\end{corollary}

Let $\mathbf{h}$ be the $\LV$-degree of the collection of sequences of hyperimmune degree, which is non-negligible by Proposition \ref{prop-non-negligible}.  

\begin{remark}\label{rmk-hyp-vd}
As shown by Kurtz, a Turing degree is hyperimmune if and only if it contains a weakly 1-generic sequence, where a sequence is weakly 1-generic if for every dense c.e.\ $S\subseteq \str$, there is some~$\sigma\prec X$ such that $\sigma\in S$.  Here $S\subseteq \str$ is called {\em dense} if every element of $\str$ has an extension in~$S$. If we write the collection of weakly 1-generic sequences as  $\mathrm{W1GEN}$ we have~$\mathbf{h}=\mathbf{deg_\mathbf{LV}(}(\mathrm{W1GEN})^{\equiv_\T }\mathbf{)}$.  
\end{remark}

An additional characterization of $\mathbf{h}$ can be given in terms of the collection $\mathrm{KR}$ of Kurtz random sequences.

\begin{prop}
$\mathbf{h}=\mathbf{deg_\mathbf{LV}(}(\mathrm{KR})^{\equiv_\T }\mathbf{)}$.
\end{prop}

\begin{proof}
($\leq_\LV$:) 
 Since every weakly 1-generic sequence is Kurtz random, by Lemma~\ref{lem-negfacts}~(ii) we have
\[\mathbf{deg_\mathbf{LV}(}(\mathrm{W1GEN})^{\equiv_\T }\mathbf{)}\leq_\LV\mathbf{deg_\mathbf{LV}(}(\mathrm{KR})^{\equiv_\T }\mathbf{)}.\]

\noindent ($\geq_\LV$:) We need to show that the collection of Kurtz random sequences that do not have hyperimmune degree is negligible.
As shown by Yu in unpublished work (see Downey and Hirschfeldt~\cite[Theorem 8.11.12]{DowHir10}), every Kurtz random sequence of hyperimmune-free degree is weakly $2$-random.  Since every 2-random 
sequence has hyperimmune degree, such a sequence must be weakly $2$\nobreakdash-random and not $2$-random.
By Corollary \ref{cor-strong}, the collection of weakly $2$-random sequences that are not $2$-random is negligible, from which the conclusion
follows.
\end{proof}

Since the collection of Kurtz random sequences includes every Martin-L\"of random sequence and every 1-generic sequence, we obtain the following result.

\begin{prop}\label{prop-rgh}
$\mathbf{r}<_\LV\mathbf{h}$ and $\mathbf{g}<_\LV\mathbf{h}$.
\end{prop}

\begin{proof}
Since $\MLR\subseteq\mathrm{KR}$ and $\mathrm{1GEN}\subseteq\mathrm{KR}$, by Lemma \ref{lem-negfacts} (ii) we have $\mathbf{r}\leq_\LV\mathbf{h}$ and $\mathbf{g}\leq_\LV\mathbf{h}$.
Moreover, $\mathrm{1GEN}\subseteq\mathrm{KR}\setminus\MLR$, so this latter collection is non-negligible, which implies $\mathbf{r}<_\LV\mathbf{h}$.  Similarly, ${\MLR\subseteq\mathrm{KR}\setminus\mathrm{1GEN}}$ implies 
$\mathbf{g}<_\LV\mathbf{h}$.
\end{proof}

$\mathbf{h}<_\LV\mathbf{h}\vee\mathbf{c}$, as the collection of computable sequences is disjoint from the collection of sequences of hyperimmune degree.  
In fact, $\mathbf{h}\vee\mathbf{c}$ can be identified as the top $\LV$-degree.

\begin{prop}\label{prop-hyp-top}
$\mathbf{h}\vee\mathbf{c}=\mathbf{1}$.
\end{prop}
\begin{proof}
By Proposition \ref{prop-negligible}~(iv) the collection of non-computable sequences of hyperimmune-free degree is negligible, from which the result immediately follows.  
\end{proof}
The following corollary, pointed out to the authors by Frank Stephan,  allows identifying~$\mathbf{h}$ also as the $\LV$-degree of immunity notions.
\begin{defn}{\ }
\begin{itemize}
\item[(i)] Let $\IM$~denote the collection of immune sequences, where a sequence is immune if it has no infinite computably enumerable subsets.
\item[(ii)] Let $\BI$~denote the collection of biimmune sequences, where a sequence is biimune if it and its complement are immune.
\item[(iii)] Let $\BHI$~denote the collection of bihyperimmune sequences, where a sequence is bi\-hyperimmune if it and its complement are hyperimmune.
\end{itemize}

Then set $\mathbf{i}=\mathbf{deg_\mathbf{LV}(}(\IM)^{\equiv_\T }\mathbf{)}$, 
$\mathbf{b}=\mathbf{deg_\mathbf{LV}(}(\BI)^{\equiv_\T }\mathbf{)}$, and
$\mathbf{bh}=\mathbf{deg_\mathbf{LV}(}(\BHI)^{\equiv_\T }\mathbf{)}$.
\end{defn}
\begin{corollary}\label{biimmunedegree}
	 We have $\mathbf{i}=\mathbf{b}=\mathbf{h}=\mathbf{bh}=\mathbf{c}^c$.
\end{corollary}
\begin{proof}
Let $\COMP$ denote the computable and $\HI$ denote the hyperimmune sequences. Then
\[
	(2^\omega \setminus \COMP)^{\equiv_\T }=(\IM)^{\equiv_\T }\supseteq (\BI)^{\equiv_\T } \supseteq (\BHI)^{\equiv_\T } = (\HI)^{\equiv_\T }.
\]
Here the first equality is by Dekker and Myhill~\cite{MR0099292} (see, for example, Odifreddi~\cite[item~1~on~page~498]{MR982269}), the first inequality is by definition, and the final equality is by Kurtz~\cite[Corollary 2.1]{MR716638}.
Using the definition of hyperimmunity given in terms of strong c.e.\ arrays (see, for example, Odifreddi~\cite[Definition III.3.7]{MR982269}), it is easy to see that every hyperimmune set is immune, and by applying this to both a set and its complement, we see that every bihyperimmune set is biimmune, giving the second inequality. 

Therefore, by Lemma~\ref{lem-negfacts}~(ii), we have
\[
	\mathbf{h} =\mathbf{bh}\leq_\LV \mathbf{b} \leq_\LV \mathbf{i} = \mathbf{c}^c = \mathbf{h},
\]
where the last equality is by Proposition~\ref{prop-hyp-top}.
\end{proof}

We can also conclude that there is no intermediate $\LV$-degree between $\mathbf{h}$ and $\mathbf{1}$.
\begin{corollary}
There is no $\LV$-degree $\mathbf{e}$ such that $\mathbf{h}<_\LV\mathbf{e}<_\LV\mathbf{1}$.
\end{corollary}
\begin{proof}
By Proposition \ref{prop-comp-atom}, $\mathbf{c}$ is an atom of $\D_\LV$, and by Corollary \ref{biimmunedegree}, $\mathbf{c}^c=\mathbf{h}$. It is a general fact that in Boolean algebras the complement of an atom is a co-atom, that is, an element~$\mathbf{k}$ such that there is no $\mathbf{k}'$ such that $\mathbf{k}<\mathbf{k'}<\mathbf{1}$ (see, for instance, Blyth~\cite[item~(3) on page~79]{Bly05}).
\end{proof}

Let $\mathbf{d}$ denote the LV-degree of the collection of sequences of DNC degree, which is non-negligible by Proposition~\ref{prop-non-negligible}.  Given that every Martin-L\"of random sequence is of DNC degree, we have~${\mathbf{r}\leq_\LV\mathbf{d}}$;  
Bienvenu and Patey~\cite{BiePat14} showed the strictness of the relation.
\begin{thm}[Bienvenu, Patey \cite{BiePat14}]\label{thm-bienvenu-patey}
$\mathbf{r}<_\LV\mathbf{d}$.
\end{thm}
Since $\mathbf{c}\wedge\mathbf{d}=\mathbf{0}$, we have the following corollary.
\begin{corollary}
$\mathbf{r}\vee\mathbf{c}$ and $\mathbf{d}$ are incomparable. 
\end{corollary}

We can also easily derive the following result and corollary.

\begin{prop}\label{prop-dg}
 $\mathbf{d}\wedge\mathbf{g}=\mathbf{0}$.
\end{prop}

\begin{proof}
No 1-generic sequence is of DNC degree by a result of Demuth and Ku\v{c}era~\cite{DemKuc87}, and thus the result follows from the same reasoning used in the proof of Proposition \ref{prop-r-g}~(i).
\end{proof}
\begin{corollary}
 $(\mathbf{r}\vee\mathbf{g})<_\LV (\mathbf{d}\vee\mathbf{g})$.
\end{corollary}
\begin{proof}
Using general properties of Boolean algebras (see, for example, Blyth~\cite{Bly05}), we have
\[(\mathbf{d}\vee\mathbf{g})\wedge(\mathbf{r}\vee\mathbf{g})^c=(\mathbf{d}\vee\mathbf{g})\wedge (\mathbf{r}^c\wedge\mathbf{g}^c)=((\mathbf{d}\vee\mathbf{g})\wedge \mathbf{g}^c)\wedge\mathbf{r}^c=(\mathbf{d}\wedge\mathbf{g}^c)\wedge\mathbf{r}^c=\mathbf{d}\wedge\mathbf{r}^c>_\LV\mathbf{0},\]
where the last equality is by Proposition \ref{prop-dg} and the final inequality is by Theorem \ref{thm-bienvenu-patey}.  
In~particular, $(\mathbf{d}\vee\mathbf{g})>_\LV (\mathbf{r}\vee\mathbf{g})$.
\end{proof}

In Section \ref{sec-implementing}, our new application of V'yugin's technique for building semi-measures implies that the collection of non-computable sequences that are not of DNC degree is non-negligible, which in turn implies that $\mathbf{d}\vee\mathbf{c}$ is not the top $\LV$-degree.
However, we can alternatively derive this latter fact as follows.

\begin{prop}\label{prop-d-h}
$\mathbf{d}<_\LV\mathbf{h}$.
\end{prop}\nopagebreak\begin{proof}
By Proposition \ref{prop-hyp-top}, $\mathbf{d}\leq_\LV\mathbf{h}\vee\mathbf{c}$, which implies that the collection of sequences of DNC~degree that are neither computable nor of hyperimmune degree is negligible.  But clearly no sequence of DNC~degree is computable, and thus we have 
$\mathbf{d}\leq_\LV\mathbf{h}$.  
Since every $1$-generic sequence has hyperimmune degree and is not of DNC degree, we have $\mathbf{h}\nleq_\LV\mathbf{d}$, and thus~${\mathbf{d}<_\LV\mathbf{h}}$.
\end{proof}

The following results about joins in $\D_\LV$ are immediate.

\begin{samepage}
\begin{corollary}{\ }
\begin{itemize}
\item[(i)] $\mathbf{c}<_\LV\mathbf{r}\vee\mathbf{c}<_\LV\mathbf{d}\vee\mathbf{c}<_\LV\mathbf{1}$.
\item[(ii)] $\mathbf{c}<_\LV\mathbf{g}\vee\mathbf{c}<_\LV\mathbf{d}\vee\mathbf{g}\vee\mathbf{c}$.
\item[(iii)]$\mathbf{d}\vee\mathbf{c}$, $\mathbf{g}\vee\mathbf{c}$, and $\mathbf{d}\vee\mathbf{g}$ are pairwise incomparable $\LV$-degrees.
\end{itemize}
\end{corollary}
\end{samepage}

The results of this section are summarized in Figure~\ref{figure-LV-degrees}. 

\subsection{Open questions} We conclude with the following open questions.

\begin{question}
Is $\mathbf{d} \vee\mathbf{g} = \mathbf{h}$? In particular, is $\mathbf{d} \vee\mathbf{g}\vee\mathbf{c} = \mathbf{1}$?
\end{question}

Given that $\mathbf{r}$ is a $\D_\LV$-atom, it is also reasonable to ask whether the same holds for $\mathbf{g}$.

\begin{question}
	Is $\mathbf{g}$ a $\D_\LV$-atom?
\end{question}

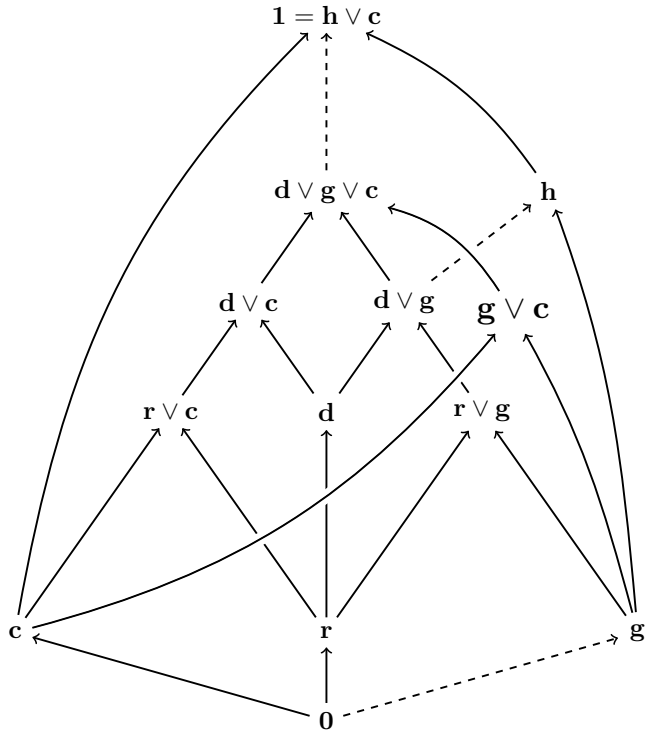
\begin{figure}[h]
\begin{center}
	\scalebox{0.9}{
		\begin{tikzpicture}[scale=0.65,auto=left,every node/.style={black}]
		
		\node (1) at (0,16) {$\mathbf{1}=\mathbf{h}\vee\mathbf{c}$};

		\node (duguc) at (0,12) {$\mathbf{d} \vee\mathbf{g} \vee\mathbf{c}$};
		\node (h) at (5,12) {$\mathbf{h}$};

		\node (duc) at (-1.75, 9.5) {$\mathbf{d} \vee\mathbf{c}$};
		\node (dug) at (1.75, 9.5) {$\mathbf{d} \vee\mathbf{g}$};
		\node (guc) at (4.2, 9.25) {\scalebox{1.25}{$\mathbf{g} \vee\mathbf{c}$}};

		\node (ruc) at (-3.5,7) {$\mathbf{r} \vee\mathbf{c}$};
		\node (d) at (0,7) {$\mathbf{d}$};
		\node (rug) at (3.5,7) {$\mathbf{r} \vee\mathbf{g}$};

		\node[circle,inner sep=2.5pt] (c) at (-7,2) {$\mathbf{c}$};
		\node (r) at (0,2) {$\mathbf{r}$};
		\node[circle,inner sep=2.5pt] (g) at (7,2) {$\mathbf{g}$};
		
		\node (0) at (0,0) {$\mathbf{0}$};

		\foreach \from/\to in {
			0/r,c/ruc,r/ruc,r/d,ruc/duc,d/duc,d/dug,r/rug,g/rug,rug/dug,duc/duguc,dug/duguc}
		\draw [->,thick] (\from) -- (\to);

		\foreach \from/\to in {
			duguc/1,dug/h,0/g}
		\draw [->,thick,dashed] (\from) -- (\to);

		\draw [->,thick,looseness=1] (c) to [out=80,in=225] (1);
		\draw [->,thick,looseness=1] (g) to [out=95,in=290] (h);
		\draw [->,thick,looseness=1] (g) to [out=105,in=299] (guc);
		\draw [looseness=1,draw=white,double=black,double distance=\pgflinewidth,line width=2.5pt] (guc) to [out=123,in=-15] (duguc);
		\draw [->,thick,looseness=1] (guc) to [out=123,in=-15] (duguc);
		\draw [looseness=1,draw=white,double=black,double distance=\pgflinewidth,line width=2.5pt] (c) to [out=15,in=230] (guc);
		\draw [->,thick,looseness=1] (c) to [out=15,in=230] (guc);
		\draw [->,thick,looseness=1] (h) to [out=125,in=-25] (1);
		\draw [->,thick] (0) -- (c);
		
		\end{tikzpicture}}
\end{center}
\caption{Standard arrows represent strict separations in the LV-degrees. Dotted arrows represent the following open questions: (a)~Is $\mathbf{g}$ a $\D_\LV$-atom? (b)~Is $\mathbf{d} \vee\mathbf{g} = \mathbf{h}$, and thus is $\mathbf{d}\vee\mathbf{g} \vee\mathbf{c} = \mathbf{1}$?}
\label{figure-LV-degrees}
\end{figure}

For the definitions of the notions appearing in the following open questions, see a standard reference such as Downey and Hirschfeldt~\cite{DowHir10}.

\begin{samepage}
\begin{question}\quad
	\begin{itemize}
		\item What are the $\LV$-degrees of the collections of sequences that are Turing equivalent to some sequence of Hausdorff dimension $1$, of packing dimension $1$, of Hausdorff dimension $<1$, of packing dimension $<1$?
		Given some $\alpha \in (0,1)$, what are the $\LV$-degrees of the collections of sequences that are Turing equivalent to some sequence of Hausdorff dimension $\alpha$ or of packing dimension $\alpha$? 
		\item What is the $\LV$-degree of the collection of sequences that {\em compute} some $1$-generic sequence?
		\item What is the $\LV$-degree of the collection of generalized low$_{\,1}$ sequences, that is, sequences~$X$ with the property that $X' \equiv_\T X \oplus \emptyset'$?
	\end{itemize}
\end{question}
\end{samepage}

\section{How to Build a Semi-measure}\label{sec-building}

In this section, we outline a template for building left-c.e.\ semi-measures that was developed~\cite{Vyu82} and applied~\cite{Vyu08,Vyu09,Vyu12} by V'yugin and which has several applications in the study of $\D_\LV$ as well as the study of $\Pi^0_1$ classes.  The main idea of V'yugin's construction is that a semi-measure on $\str$ can be seen as a network flow on a directed graph $G$ such that 
\begin{itemize}
\item[(i)] the nodes of $G$, $\V_G$, are the elements of $\str$, and 
\item[(ii)] the edges of $G$, $\E_G$, are pairs $(\sigma,\tau)$ of nodes $\sigma,\tau\in\str$ such that $\sigma\prec \tau$.
\end{itemize}

For $\sigma,\tau\in \str$ with $\sigma \preceq \tau$ we will say that $\sigma$ is {\em above} $\tau$ and that $\tau$ is {\em below} $\sigma$; that is, in this article the binary tree~$\str$ grows downward. Note that, while this goes against the usual convention in computability theory, it has the intuitive advantage that measure will flow from the root $\varepsilon$ downwards, as liquids naturally do.

Given $\sigma,\tau\in \str$ with $\sigma \prec \tau$, the length of $(\sigma,\tau)$, written as $|(\sigma,\tau)|$, is defined to be $|\tau|-|\sigma|$. If $|(\sigma,\tau)|=1$ then we always have
$(\sigma,\tau) \in \E_G$; such edges of $G$ will be referred to as \emph{normal edges} and the set of normal edges will be denoted by $\mathcal{N}_G$. 
If $|(\sigma,\tau)|>1$ then $(\sigma,\tau)$ may or may not be in $\E_G$; if it is, we call~$(\sigma,\tau)$ an \emph{extra edge} of $G$. The set of extra edges will be denoted by $\mathcal{X}_G$. We will omit the subscripts if $G$ is clear from context.

Directed graphs $G$ that satisfy $\V_G=\str$ as described above will be called \emph{$\str$-digraphs}.  In the sequel, we will restrict our attention to computable $\str$-digraphs.

\begin{defn}
Given a $\str$-digraph $G$, a \emph{network} on $G$ is a function $q\colon \E_G\rightarrow\mathbb{Q}\cap [0,1]$ satisfying, for each $\sigma\in\str$,
\[
\sum_{(\sigma,\tau)\in \E_G}q(\sigma,\tau)\leq 1.
\]
\end{defn}

\noindent The idea here is that for a node $\sigma$, $q(\sigma,\tau)$ gives the proportion of the flow arriving in $\sigma$ that continues to flow into~$\tau$. 

In the remainder of the article, we will always have $q(\sigma,\tau)>0$ for every extra edge $(\sigma,\tau)\in \X$. In fact, if $|(\sigma,\tau)|>1$, we will silently identify the two properties $q(\sigma,\tau)=0$ and $(\sigma,\tau)\notin \E$ since both cases equally have  no effect on the outcome  of the construction. Note however that for normal edges $(\sigma,\tau)\in \N$ the case $q(\sigma,\tau)=0$ will occur quite often.

\begin{defn}
The \emph{amount of flow into a node $\tau$}, denoted $R(\tau)$, is defined inductively by
\begin{align*}
R(\varepsilon)&=1,\\
R(\tau)&=\sum_{(\sigma,\tau)\in \E_G}q(\sigma,\tau)R(\sigma).
\end{align*}
\end{defn}

\noindent
Hereafter we will refer to $R$ as the \emph{in-flow function associated to $q$}.  Observe further that if $q$~is computable, then so is $R$.  

\begin{remark}\label{rmk-R}
$\sigma\prec\tau$ does not necessarily imply that $R(\sigma)\geq R(\tau)$.  In particular, not all of the flow that we observe below $\sigma$ must have flowed through $\sigma$ itself, as there could be an extra edge that bypasses $\sigma$ and diverts flow to an extension of $\sigma$.
\end{remark}

To correct for this lack of monotonicity of $R$, we define the $q$-flow associated with a network~$q$.  Given $\sigma\in\str$, let $T_\sigma$ \label{dkhfdghjkwqeuzidgfbnsdf} be the collection of finite prefix-free sets of strings $\tau$ such that $\sigma\preceq \tau$.

\begin{defn}\label{def:q-flow}
Let $q$ be a network on a $\str$-digraph $G$, and let $R$ be the in-flow function associated to $q$.  Then the \emph{$q$-flow} $P$ is defined by
\[
P(\sigma)=\sup_{D\in T_\sigma}\sum_{\tau\in D}R(\tau).
\]
\end{defn}

\noindent
$P(\sigma)$ is thus the maximal amount of flow that can be observed passing through a set of extensions of the node~$\sigma$. The motivation for looking at prefix-free sets $D$ of nodes is to avoid counting the same quantity of flow more than once. 
Note that since $\{\sigma\}\in T_\sigma$, we always have $P(\sigma) \geq R(\sigma)$, but equality need not hold due to the reason discussed in~Remark \ref{rmk-R}.

We have the following important fact. 

\begin{lem}\label{lem:semimeasure}
Let $q$ be a computable $\str$-digraph.  Then the $q$-flow $P$ is a left-c.e.\ semi-measure.
\end{lem}

\begin{proof}
Clearly, $P(\varepsilon)= 1$.   Let $s_0=\sup_{D\in T_{\sigma0}}\sum_{\tau\in D}R(\tau)$ and $s_1=\sup_{D\in T_{\sigma1}}\sum_{\tau\in D}R(\tau)$. Given $\delta>0$, there are $D_0\in T_{\sigma0}$ and $D_1\in T_{\sigma1}$ such that 
\[
\sum_{\tau\in D_i}R(\tau)\geq s_i-\delta/2
\]
for $i=0,1$.  Then $D_0\cup D_1\in T_\sigma$, and hence
\[
\sup_{D\in T_\sigma}\sum_{\tau\in D}R(\tau)\geq\sum_{\tau\in D_0\cup D_1}R(\tau)\geq s_0+s_1-\delta,
\]
for every $\delta>0$. 
 Thus $P(\sigma)\geq P(\sigma0)+P(\sigma1)$.  Lastly,  $P(\sigma)$ is left-c.e.\ uniformly in $\sigma$, as $G$, $q$, and~$R$ are all computable.
\end{proof}

\begin{defn}
A network $q$ is \emph{elementary} if $q(\sigma,\tau)=1/2$ for all but finitely many $(\sigma,\tau)\in \N$.
\end{defn}

By the definition of a network $q$, it follows that the set of extra edges $\X$ is finite if $q$~is elementary. Since by definition networks~$q$ only take rational values, every elementary network~$q$ is computable.   Given a computable network $q$, we can write~$q$ as a limit of elementary networks~${(q_n)_{n\in\omega}}$ by requiring that
\begin{itemize}\label{elementary_approx}
\item[(i)] $q_n(\sigma,\tau)=q(\sigma,\tau)$ if $|\tau|\leq n$;
\item[(ii)] $q_n(\sigma,\tau)=1/2$ if $(\sigma,\tau)\in \N$ and $|\tau|>n$;
\item[(iii)] $q_n(\sigma,\tau)=0$ if $(\sigma,\tau)\in \X$ and $|\tau|>n$.
\end{itemize}

Note that these conditions imply that $q_{n-1}$ and $q_n$ agree on every edge $(\sigma,\tau)$ except possibly on edges $(\sigma,\tau)$ satisfying $|\tau|=n$.
We refer to such a sequence of elementary networks as the \emph{sequence of elementary restrictions} of $q$.  Moreover, we will refer to each $q_n$ as the \emph{level $n$ elementary restriction} of~$q$.

\subsection{The general template}\label{subsec-template}
The semi-measure $P$ that we construct will be one induced by a network flow $q$ as described in the previous paragraphs. Here, $q$ will be constructed through an infinite procedure which works in stages. At each stage $n$, an elementary network~$q_n$ together with its extra edge set $\X_n$ will be built. In the end we will then let $q=\lim_n q_n$ and $\X=\bigcup_n \X_n$. We first make some general informal remarks about the overall procedure, and then go on to describe in formal detail the individual stages.

\bigskip

\noindent The general construction template depends upon three parameters:
\begin{itemize}
\item[(1)] A computable function $\task\colon\omega\rightarrow\omega$, called the \emph{task function}, such that the values \[\task(0),\task(1),\task(2),\task(3),\dotsc\] follow the pattern
\[
0,1,0,1,2,0,1,2,3,0,1,2,3,4,\dotsc
\]
In particular, for each $i$, the set $\{n\colon\task(n)=i\}$ is infinite and $\task(n)\neq\task(n+1)$ for every~$n$.  Every node will be assigned a task; namely, each $\sigma\in\str$ will be assigned the task~$\task(|\sigma|)$.  For a given task $i$, the \emph{$i$-nodes} are the nodes $\sigma\in\str$ with~$\task(|\sigma|)=i$.
\item[(2)] A computable predicate $B(q',\sigma,\tau)$ which is defined for \textit{elementary} networks $q'$ on a $\str$-digraph $G$ and strings $\sigma$, $\tau$ such that both are $i$-nodes for the same $i\in\omega$.
\item[(3)] A computable, strictly increasing function $\ct\colon\omega\rightarrow\omega$.
\end{itemize}
The predicate $B$ will be determined by the requirements we are attempting to satisfy, while the function $\ct$ will be specifically used to provide the initial values for countdowns to expiration for certain nodes that are active in the construction, in a technical sense to be explained shortly.

We take action towards fulfilling the task $i$ if we add an extra edge connecting two $i$-nodes; we will refer to such an edge as an \emph{$i$-edge} (or as an edge that is \textit{assigned to} task $i$).  That is, an edge~$(\sigma,\tau)\in \E_G$ is an $i$-edge if $\task(|\sigma|)=\task(|\tau|)=i$.  Let $\X[i]$ be the set of extra edges assigned to task $i$.  Note that we never assign normal edges to any task $i$, since $\task(n)\neq\task(n+1)$ for every~$n$.

In the course of the construction, for $j<i$, we would ideally want to first perform all actions necessary for task $j$ before beginning to work on task $i$. That is, 
for every extra edge~$(\sigma,\tau)$ between a pair of $i$-nodes $\sigma$ and $\tau$ and for any extra edge $(\sigma',\tau')$ assigned to some task~$j$ with~$j<i$, we would like to have $|\tau'|<|\sigma|$.
But, in fact, during the construction we will not be able to always ensure this property.
After having added $(\sigma,\tau)$ for task $\task(|\sigma|)=\task(|\tau|)=i$ it may turn out later in the construction that further edges for task~$j$ need to be added. Adding them will then invalidate our previous actions for task~$i$. The edge $(\sigma,\tau)$ stays in the digraph, but we will consider it a failure\label{sdfsdsdfsdfsdgdhfgh}, as it does not help us achieve the desired goal for task $i$. While the presence of $(\sigma,\tau)$ also causes no harm, we will, at some later stage, have to completely restart the construction for task $i$. The construction can therefore be thought of as a type of finite injury argument.

For a given task $i$, we will need to talk about the minimal length of an $i$-node to which   an extra edge can be attached.  We thus define the following auxiliary function~$w$:  Let $q'$ be an elementary network on $G$, with the associated set of extra edges $\X'$ through which some of the flow passes.  Then for each $i\in\omega$, we define
\[
w(i,q')=\min\{n\colon\task(n)=i \;\wedge\; (\forall j<i)(\forall(\sigma,\tau)\in \X'[j])\; |\tau|<n\}.
\]
That is, $w(i,q')$ is the least $n$ such that (i)~$\task(n)=i$ and (ii)~every edge in $G$ assigned to task~$j$ for some $j<i$ ends in a node of length less than $n$.  

For an arbitrary (that is, not necessarily elementary) computable network~$q'$, $w(i,q')$ may be undefined in general. But in fact, for $q$'s built using the template described here,
$w(i,q')$ will always be well-defined by the above equation, and $w(i,q)=\lim_n w(i,q_n)$, where $q_n$ is the level~$n$ elementary restriction of $q$. 
The lengths $n$ where $w(i,q_{n-1})\neq w(i,q_n)$ correspond to the \qt{failures} described in the previous paragraph.

Another component of our construction is that at each stage, a number of nodes may be set as active, serving as candidates to which an extra edge may be attached.  
Before activation, all flow into a node $\sigma$ will be equally divided to flow into $\sigma$'s direct successor nodes $\sigma0$ and~$\sigma1$ through the corresponding  normal edges. To {\em activate} a node we reduce the flow from $\sigma$ into~$\sigma0$ and~$\sigma1$, resulting in a certain amount of flow into $\sigma$ being temporarily unused. We say that we have {\em delayed} part of the flow. In a later step we may then attach an extra edge to $\sigma$ and direct the delayed, leftover flow through this new edge.

More formally, for an elementary network $q'$, we have a function $d'$, called a \emph{flow-delay function}, which satisfies
\[
d'(\sigma)=1-q'(\sigma,\sigma 0)-q'(\sigma,\sigma1)
\]
for every $\sigma\in\str$.  This is precisely the proportion of flow into $\sigma$ that is prevented from flowing into~$\sigma0$ and~$\sigma1$.  The \emph{active} nodes consist of those nodes $\sigma$ such that $0<d'(\sigma)<1$; the construction will be such that if we block all of the flow through a node $\sigma$ by setting $d'(\sigma)=1$, then it, and all of its extensions, will never be activated from that point on.  Moreover, for~$j<i$, to enforce the requirement that all $j$-edges end before any $i$-edges begin, whenever we attach an extra edge to a $j$-node $\tau$, all active $i$-nodes whose length is less than $|\tau|$ become unusable as, by the conditions in the construction, we will never attach edges to such nodes. We will call such nodes {\em deactivated}. Intuitively, we can then think of the flow that was delayed at such nodes as wasted.

Next, given a node $\sigma$ to which we would like to attach an extra edge, there is a function~$\beta(\sigma,q',n)$ that selects (somewhat arbitrarily) a candidate $\tau\succ \sigma$ of length $n$ for connecting an edge between $\sigma$ and $\tau$ in an elementary network $q'$ in such a way as to satisfy the predicate~$B$, if such~a~$\tau$ exists.  Specifically,
\[
\beta(\sigma,q',n)=\min\{\tau\in2^n \colon \tau \succ \sigma \;\wedge\; \task(|\sigma|)=\task(|\tau|) \;\wedge\; B(q',\sigma,\tau)\},
\]
where the minimum refers to the lexicographic ordering of strings $\tau$ of length $n$.

\bigskip

After these informal remarks, we describe in detail how to construct the network $q$, with its set of extra edges $\X$ and its flow-delay function $d$: As mentioned at the start of this subsection, we first build a sequence $(q_n,\X_n,d_n)_{n\in\omega}$, where each $q_n$ is an elementary network with 
associated set of extra edges $\X_n$. For each $n\in \omega$ we will let $d_n$ denote the flow-delay function associated with~$q_n$. In the end we will set $q= \lim_nq_n$,
$\X=\bigcup_n \X_n$ and $d=\lim_n d_n$.  

\smallskip

The definition of the sequence $(q_n, \X_n, d_n)_{n\in\omega}$ proceeds in stages as follows:  For $n=0$, 
\[
q_0(\sigma,\tau)=\begin{cases}
1/2 & \text{if $\tau=\sigma0$ or $\tau=\sigma1$,}\\
0 & \text{otherwise}.
\end{cases}
\]
Clearly $d_0(\sigma)=0$ for all $\sigma\in\str$ and $\X_0=\emptyset$. 

Suppose we have defined $(q_{n-1},\X_{n-1},d_{n-1})$, where for all $(\sigma,\tau)\in \X_{n-1}$, $|\tau|<n$.  We will first define $\X_n, d_n$, and then $q_n$.  The goal of this stage of the construction is to attach an extra edge connecting a $\task(n)$-node whose length is strictly less than $n-1$ with a $\task(n)$-node of length~$n$. We consider two cases.

\bigskip
\noindent
{\bf Case 1:}  $w(\task(n),q_{n-1})=n$.  This means that the extra edges in $\X_{n-1}$ assigned to some task~$j<\task(n)$ terminate in nodes of length $\leq n-1$, and this is the least $n$ for which this holds.  This further implies that there is no active $\task(n)$-node of length less than $n$ to which we can attach an extra edge.  We thus take the following steps:
\begin{itemize}
\item[(i)]  Set $\X_n=\X_{n-1}$.
\item[(ii)] Set $d_n(\sigma)=
\begin{cases}
1/\ct(n) & \text{if  $|\sigma|=n$}\\
d_{n-1}(\sigma) & \text{otherwise.}\end{cases}$
\end{itemize}
  Setting $d_n(\sigma)=1/\ct(n)$ for each $\sigma$ of length $n$ has the effect of activating these nodes, in anticipation of attaching a $\task(n)$-edge to them later in the construction.  We  call this the \emph{initial activation} of the nodes, since Case~1 is the case where we begin working towards task $i$ (or where we begin \textit{anew} to work towards it, in case that a previous injury has occured for task~$i$ that requires us to start over).
  
  Recall that $\ct$ provides the initial value for a countdown mechanism that we will use during the construction; once we implement this template for a specific application, we will have to choose~$\ct$ carefully to ensure that a positive amount of flow stays in the network in the limit.

\bigskip
\noindent
{\bf Case 2:} $w(\task(n),q_{n-1})<n$.  Our hope in this case is that we can attach some extra edges from $\task(n)$-nodes of length $\geq w(\task(n), q_{n-1})$ to $\task(n)$-nodes of length $n$.  Thus we search for $\sigma\in\str$ such that the following four conditions hold:
\begin{itemize}
\item[$(a)$] $w(\task(n),q_{n-1})\leq|\sigma|<n$;
\item[$(b)$] $0<d_{n-1}(\sigma)<1$; 
\item[$(c)$] $\beta(\sigma,q_{n-1},n)$ is defined; and
\item[$(d)$] $\sigma\prec \rho$ implies that $(\sigma,\rho)\notin \X_{n-1}$.
\end{itemize}
Condition $(a)$ guarantees that the start of the new edge occurs beyond the end of any currently present
$j$-edge for $j<\task(n)$; in particular, this rules out attaching edges at deactivated nodes.  Condition $(b)$ guarantees that $\sigma$ is active (henceforth, we will refer to a node $\sigma$ such that $0<d_n(\sigma)<1$ as \emph{active at stage $n$}).  Condition $(c)$ guarantees that $\sigma$ is assigned to task $\task(n)$ and that there is a length~$n$ node that can serve as the endpoint of a new $\task(n)$-edge we want to attach at $\sigma$ (that is, the predicate $B$ is satisfied).  Finally, condition $(d)$ guarantees that no extra edge has been previously attached starting at $\sigma$.

Let $\C_n$ be the set of $\sigma\in\str$ such that conditions $(a)$--$(d)$ are satisfied.  Then we have two subcases to consider.

\bigskip
\noindent
{\bf Subcase 2.1:}  $\C_n\neq\emptyset$.  For every $\sigma \in \C_n$ and every $\tau \succ \sigma$ with $|\tau|=n$ we let 
\[
d_n(\tau)=
	\left\{
		\begin{array}{lll}
			0 & \mbox{if } \tau=\beta(\sigma,q_{n-1},n), \\
			d_{n-1}(\sigma)/(1-d_{n-1}(\sigma)) & \mbox{else}.
		\end{array}
	\right.
\]
For all other $\tau$ we let $d_n(\tau) = d_{n-1}(\tau)$.  

By condition~$(b)$ above, setting $d_n(\tau)=0$ makes $\tau$ inactive, meaning that we will not add any further $\task(n)$-edges to any extensions of $\tau$, with one possible exception:  
It may be that at a later stage $n'>n$ with $\task(n') < \task(n)$, a new $\task(n')$-edge is added, which at some even later stage $n''>n'>n$ with $\task(n'') = \task(n)$ would lead to Case~1 above occurring again for task~${\task(n'') = \task(n)}$.
This would in turn lead to all strings of length $n''$ getting initially activated for task $\task(n'') = \task(n)$ at stage $n''$ where we begin anew to work on that task. In this case we say that task $\task(n'') = \task(n)$ has been {\em injured} by task $\task(n')$.

When a node is assigned a non-zero delay value by the second line above, we call this its {\em non-initial activation}. This is because that new delay value at node $\tau$ is a consequence of an earlier  assignment of a non-zero delay value to the strictly shorter node $\sigma$.

Note that when initially activating a node $\sigma$, we assign a delay of the form $1/k$, where $\ct(|\sigma|)=k$ for some $k\in\omega$.  Moreover, the mapping $d\mapsto d/(1-d)$ used for non-initial activations maps such a number to $1/(k-1)$, which is then mapped to $1/(k-2)$, and so on.  Note further that the nodes where these new delay values are set are by construction $\task(n)$-nodes. The reciprocal of these assigned delay values are positive integers, and we can interpret them as a counter counting down  by $1$ along a path every time a $\task(n)$-edge branches off it; see Figure~\ref{fig_adding_edges}.

Even on the same path different tasks are initially activated separately at different nodes of appropriate length. The countdown happens separately for all tasks, as a new delay value assigned to an $i$-node depends on the delay value of an $i$-node of shorter length, and not on the delay values of $j$-nodes with $i\neq j$. It  therefore makes sense to talk about the \emph{$i$-counter} for task $i$ along a given path, and we will use this expression in the informal explanations in the sequel.

As we continue to add edges for task $\task(n)$ that branch off a path, the $\task(n)$-counter may eventually reach $1$ on some initial segment of that path. Such a node is then by definition inactive. (Formally, the counter reaching $1$ means that the delay value along the path has increased until all flow is blocked at a value of~$1$.) Once this happens, by construction, we stop attaching $\task(n)$-edges on any extension of that initial segment of the path. 

\begin{figure}
  \centering
  \def\svgwidth{16cm}
   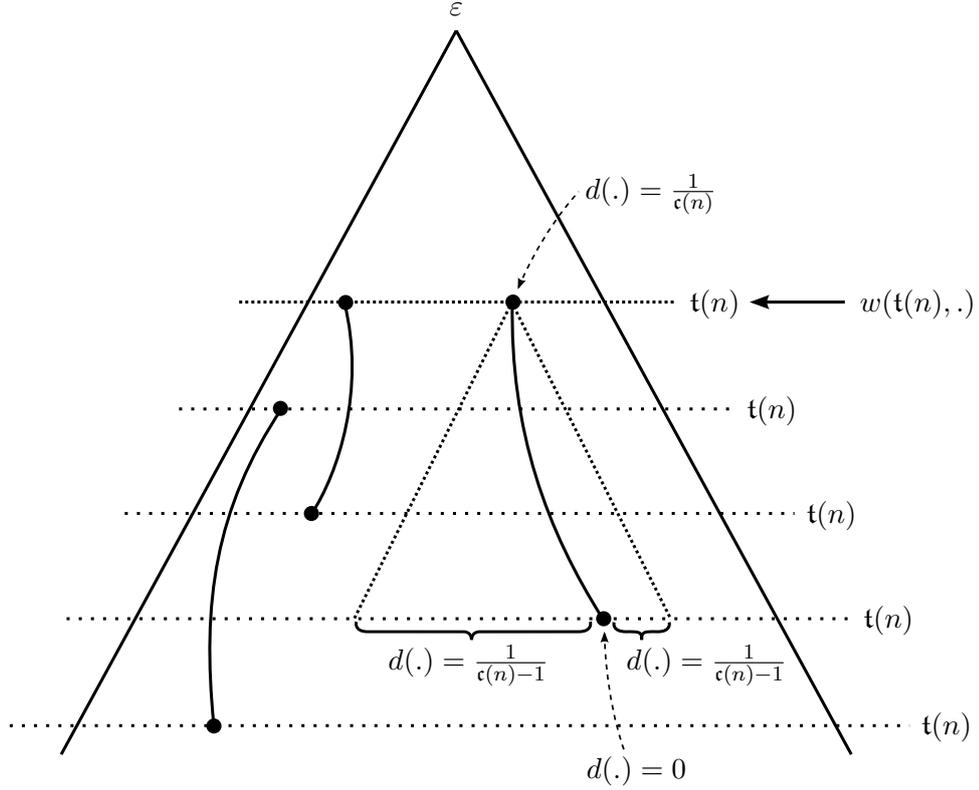
\caption{An edge for task $\task(n)$ is added. The root of the edge was initially activated with counter value $\ct(n)$. The node at the end of the new edge has delay value~$0$, thus is inactive. All other extensions of the root receive a positive delay value by non-initial activation, and thereby become active. The counter value on these nodes, which is the reciprocal of the value of $d$ on the respective node, has been reduced by $1$ compared to the counter value on the root.
Note how other, completely independent $\task(n)$-edges can occur off to the side.}
\label{fig_adding_edges}
\end{figure}

\smallskip

Next, we set
\[
\X_n=\X_{n-1}\cup\{(\sigma,\beta(\sigma,q_{n-1},n))\colon\sigma\in \C_n\}
\]
and
\[
q_n(\sigma,\tau)=
	\left\{
		\begin{array}{lll}
			\frac{1}{2}(1-d_n(\sigma)) & \mbox{if } \tau=\sigma0\;\text{or}\;\tau=\sigma1,\\
			d_{n}(\sigma) & \mbox{if } (\sigma,\tau)\in \X_n,\\
			0 & \mbox{otherwise.}
		\end{array}
	\right.
\]
Note that for $j>\task(n)$, $w(j,q_k)>n$
for all $k \geq n$.  In particular, we are now prevented from attaching an edge to any $j$-nodes that were active at the beginning of this stage; as mentioned above we call  such nodes {\em deactivated}.

\bigskip
\noindent
{\bf Subcase 2.2:} $\C_n=\emptyset$.  Then we set $d_n=d_{n-1}$, $\X_n=\X_{n-1}$, and $q_n=q_{n-1}$.  No new nodes are activated, nor do any active nodes become deactivated.

\bigskip

\noindent To finalize the outline of the construction template we lastly set $q=\lim_n q_n$, $d=\lim_n d_n$, and $\X=\bigcup_n \X_n$. 
It is not difficult to check that $q$ and $d$ are computable functions and that $\X$ is a computable set. It then follows from Lemma \ref{lem:semimeasure} that the resulting $q$-flow $P$, as in Definition~\ref{def:q-flow}, is a left-c.e.\ semi-measure.

\subsection{Verification of the general template}

We now work to establish the desired properties of the constructed objects $q$, $d$, $\X$, $R$, $P$, and so on. 
For the sake of notational simplicity, during this verification, we will again use the letters $q_n$, $d_n$, and $\X_n$, for $n\in\omega$, to refer to the finite approximations of $q$, $d$, and $\X$ that we built in the previous subsection. In particular note that, for~$q$, these finite approximations~$q_n$, $n\in \omega$, coincide with the sequence of elementary restrictions discussed on page~\pageref{elementary_approx}.

\bigskip

Before we implement this template, we show that a number of features of the construction can be established independently of the concrete implementation. First,
the following two lemmata show that the construction prevents certain relative arrangements of extra edges.
\begin{lem}\label{djsdjhkdfgjhqweshfasjddbwe}
	Assume that an extra edge $(\zeta,\xi)$ is added during the construction. Then no node~$\sigma$ such that $\task(\zeta)=\task(\sigma)$ and $\zeta \prec \sigma$ and $|\sigma|\leq |\xi|$ can ever become active during the construction.
\end{lem}
\begin{proof}
	The fact that $(\zeta,\xi)$ is added implies that $\zeta$ was activated at stage $|\zeta|$ and cannot have been deactivated until after stage $|\xi|$.
	This implies in particular that between stages~$|\zeta|$ and~$|\xi|$ no injury of task $\task(\zeta)$ occured, so that~
	$\sigma$ cannot have been initially activated at stage~$|\sigma|$.
	Assume that $\sigma$ was activated non-initially. 
	Then there must exist a sequence of extra $\task(\zeta)$-edges
	\[(\nu_1, \mu_1), (\nu_2,\mu_2), (\nu_3,\mu_3), \dots, (\nu_\ell, \mu_\ell)\]
	such that for all $2 \leq i \leq \ell$ we have
	\[\zeta=\nu_1, \quad  \nu_{i-1} \prec \nu_i,\quad |\mu_{i-1}|=|\nu_i|,\quad   \mu_{i-1}\neq\nu_i,\quad |\mu_\ell| = |\sigma| \quad \text{and} \quad 
	\mu_\ell \neq \sigma,\]
	 see Figures~\ref{fig_adding_edges} and~\ref{edges_along_a_path}. But, by condition~(d) in Case~2 of the construction, the presence of the edge~$(\zeta, \mu_1)$ 
	would have precluded~$(\zeta,\xi)$ from having been added later, a contradiction.	
\end{proof}
\begin{lem}\label{fact_constr}
There do not exist strings $\zeta \prec \sigma \prec \xi$ and $\zeta \prec \sigma \prec \tau$ such that $|\xi|\leq|\tau|$ and $(\zeta,\xi) \in \X$ and $(\sigma,\tau) \in \X$. 
\end{lem}
Note that without the condition ``$|\xi|\leq|\tau|$'' the statement is false, as typically there are many pairs of extra edges
for which that condition does not hold but which satisfy the other conditions in the statement.
\begin{proof} 
	Assume for a contradiction that extra edges $(\zeta,\xi)$ and $(\sigma,\tau)$ as in the statement exist.

	First assume that both are assigned to the same task; that is, $\task(\zeta)=\task(\xi)=\task(\sigma)=\task(\tau)$. 
	Then by Lemma~\ref{djsdjhkdfgjhqweshfasjddbwe}, the fact that $(\zeta,\xi)$ was added implies that 
	$\sigma$~is never activated, and thus 
	$(\sigma,\tau)$~cannot have been added, 
 which contradicts our initial assumption.

	Thus, $(\zeta,\xi)$ and $(\sigma,\tau)$ must be assigned to two different tasks. 
	(In particular, in this case, the strict inequality $|\xi|<|\tau|$ must hold.) We distinguish two cases.

	If $\task(\zeta)=\task(\xi) < \task(\sigma) =\task(\tau)$, then even if  $\sigma$~was activated at stage $|\sigma|$, the addition of $(\zeta,\xi)$ would have constituted an injury of task $\task(\sigma)$ that would have led to $\sigma$ becoming deactivated, which means that $(\sigma,\tau)$ could not have been added later, a contradiction.

	If, on the other hand, $\task(\sigma)=\task(\tau) < \task(\zeta)=\task(\xi)$, then the fact that $(\sigma,\tau)$ was added implies that $\sigma$ must have been activated at stage $|\sigma|$. This could be for two possible reasons:

	Either $w(\task(|\sigma|),q_{|\sigma|-1})=|\sigma|$ held at stage $|\sigma|$ and thus $\sigma$ was initially activated as described in Case~1. The cause for that new initial activation, namely the addition of some edge for some task  $j < \task(\sigma)$, would have also deactivated $\zeta$, since 
	$j < \task(\sigma) < \task(\zeta)$, and that would have precluded the addition of $(\zeta,\xi)$ later, again a contradiction.

	Or, if $\sigma$ was non-initially activated due to Subcase~2.1, then that must have been due to an extra edge $(\nu,\mu)$ for task $\task(\sigma)$ with $|\mu|=|\sigma|$ having been added at stage $|\sigma|$. Then the addition of that edge would have injured $\task(\zeta)$, which again would have deactivated $\zeta$, yet another contradiction.
	
\medskip
	
	In summary, there is no scenario that allows the presence of both  $(\zeta,\xi)$ and $(\sigma,\tau)$ in the digraph at the same time.	
\end{proof}

Next, we establish that the work towards every individual task terminates eventually.

\begin{lem}[Stability Lemma]\label{finite_i_edges}
For every $i\in\omega$, $\X[i]$ is finite and $w(i,q)<\infty$.
\end{lem}

\begin{proof}
First, observe that if $\X[j]$ is finite for every $j<i$, then $w(i,q)<\infty$. It is therefore sufficient to prove the first part of the statement.

So suppose that $i$ is minimal such that $\X[i]$ is infinite. Then by the previous observation we have~$w(i,q)<\infty$. For $\sigma$ with $|\sigma|\geq w(i,q)$, define $m_\sigma$ to be the maximal $m>w(i,q)$ such that there is an edge ${(\sigma\uh m,\tau) \in \X[i]}$ where $\tau$ is incomparable with $\sigma$. 
If no such $m$ exists, set~${m_\sigma = w(i,q)}$. 

Then define a function $u$ via
\[
u(\sigma)=\begin{cases}
1/d(\sigma\uh m_\sigma) &  \text{if } d(\sigma\uh m_\sigma)>0, |\sigma|\geq w(i,q), \\
\ct(w(i,q)) & \text{if } |\sigma| < w(i,q).\\
\end{cases}
\]
Note that these two cases are exhaustive; to see this assume that $|\sigma|\geq w(i,q)$. If $m_\sigma=w(i,q)$, then by construction $d(\sigma \uh m_\sigma) = 1/\ct(m_\sigma) >0$. The only other possibility is that $m_\sigma$ is 
the maximal $m>w(i,q)$ such that there is an edge ${(\sigma\uh m,\tau) \in \X[i]}$ where $\tau$ is incomparable with~$\sigma$. But then $d(\sigma \uh m_\sigma) > 0$ as well, as otherwise the edge $(\sigma\uh m_\sigma,\tau)$ would not have been added according to the conditions in the construction.

We claim that $u(\sigma)$ is an upper bound on the number of possible $i$-edges branching off below length $\max(w(i,q),|\sigma|)$ from any path going through $\sigma$.

First, consider $\sigma$ meeting the conditions of the first line of the definition, and such that an edge $(\sigma\uh m_\sigma,\tau)$ as in the definition of $m_\sigma$ exists.  Since by the choice of $m_\sigma$ the edge~$(\sigma\uh m_\sigma,\tau)$ is the last edge branching off above $\sigma$, and by the discussion of the $i$-counter mechanism above, we know that then at most $\frac{1}{d(\sigma\uh m_\sigma)}-2$ further $i$-edges can branch off below~$\sigma$ from any path extending $\sigma$, and the claim in this case follows.

Secondly, consider $\sigma$ meeting the conditions of the first line of the definition, but where an edge of the form $(\sigma\uh m_\sigma,\tau)$ as in the definition of $m_\sigma$ does {\em not} exist. For those $\sigma$ we have that a parent~$\rho$ of~$\sigma$ with $|\rho|=w(i,q)$ has been initially activated, but that there is no extra $i$-edge that branches off between $\rho$ and~$\sigma$. Again by the discussion of the $i$-counter mechanism, we know that then at most $\ct(w(i,q)) - 1$ $i$-edges can branch off below $\sigma$ from any path extending~$\sigma$. Since 
\[u(\sigma)=1/d(\sigma\uh m_\sigma)=1/d(\sigma\uh w(i,q))=\ct(w(i,q)),\] the claim in this case follows.

Lastly, consider $\sigma$  satisfying $|\sigma|<w(i,q)$. Let $\tau \succ \sigma$ be of length~$w(i,q)$. Then $\tau$~is initially activated with $d(\tau)=1/\ct(|\tau|)$. 
By the definition of $u$, $u(\sigma)=u(\tau)$, and we can argue as in the previous paragraph to conclude that at most $\ct(w(i,q)) - 1$ $i$-edges can branch off below length~$w(i,q)$ from any path extending $\sigma$.

It should now be clear that $u$ is constant on all strings $\sigma$  with $|\sigma| \leq w(i,q)$; and that for arbitrary strings $\sigma$ and $\tau$ with $\sigma \preceq \tau$ we have $u(\sigma) \geq u(\tau)$. We then define the function~${\widehat u\colon\cs\rightarrow\omega}$ by letting, for every $A \in \cs$,
\[
\widehat u(A) = \min\{n\colon     u(A \uh n)= u(A \uh \ell) \textnormal{ for all } \ell \geq n\}.
\]
We claim that the function $\widehat u$ is continuous. This is because $(a)$ $u$ is non-increasing over longer and longer initial segments  of a path $A$, $(b)$ $u$ only takes integer, positive values, and $(c)$ a decrease in~$u$ cannot happen arbitrarily late along $A$. This last point $(c)$ follows from the two facts that {\em (i)} at every node at most one edge starts (by construction) and that {\em (ii)} for an $i$-edge branching off $A$ at $A\uh \ell$ we must have that $\ell$ is either $w(i,q)$ or the length of the endpoint of the previous $i$-edge branching off $A$; otherwise $A\uh \ell$ would not be active; see Figure~\ref{edges_along_a_path}.
Therefore, for a long enough initial segment $A \uh k$ of $A$, $u(A \uh k)$ has stabilized; meaning that $A \uh k$ already determines~$\widehat u(A)$. 

\begin{figure}
	\centering
	\def\svgwidth{16cm}
	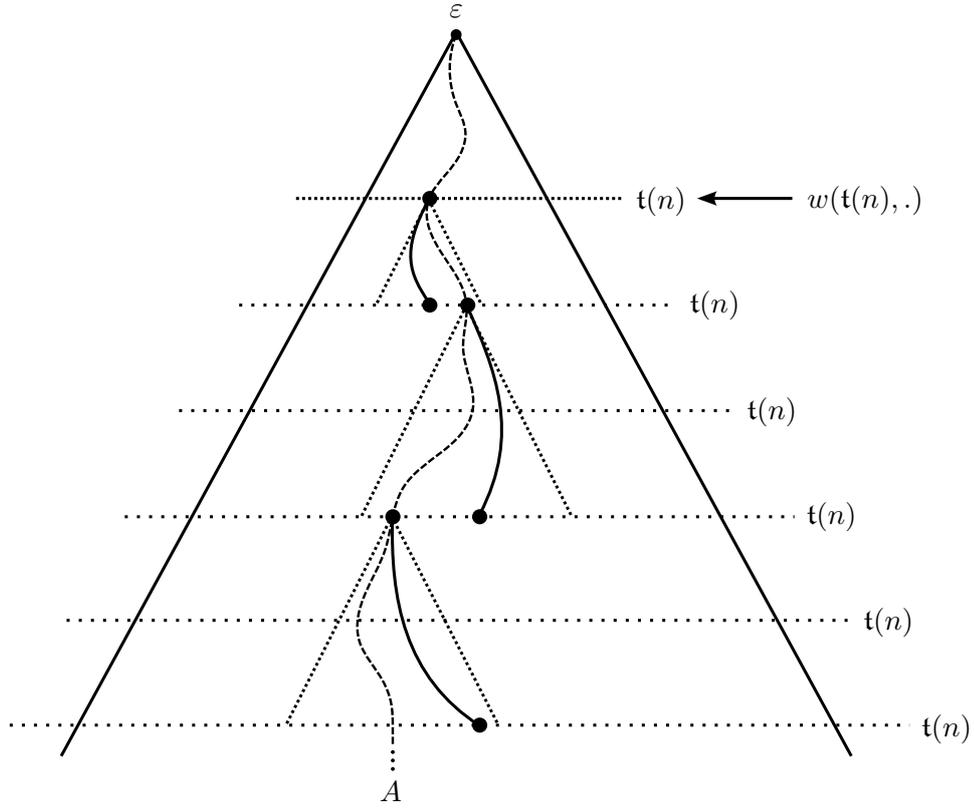
	\caption{A sequence of extra edges branching off a given path $A$. Note how the length of the start point of every edge has to coincide with the length of the endpoint of the previous edge branching off the path.}
	\label{edges_along_a_path}
\end{figure}

Because $2^\omega$ is compact, $\widehat u$ is bounded by some $N \in \omega$, meaning in particular that ${u(\sigma)=u(\sigma\uh N)}$ for all $\sigma$  with  $|\sigma|\geq N$. But then no new $i$-edge $(\sigma,\tau)$ can be attached to any such~$\sigma$, as that would imply $u(\tau) < u(\sigma)$, contradicting the choice of $N$.  Thus $\X[i]$~cannot be infinite.
\end{proof}

\begin{defn}
For a finite sequence $\sigma\in\str$ we call an infinite sequence $X\in\cs$ an \emph{$i$\nobreakdash-continuation} of $\sigma$ if $i=\task(|\sigma|)$, $\sigma \prec X$, and $B(q_{n-1},\sigma,X \uh n)$ holds for almost all $n$ with~${\task(n)=i}$. 
\end{defn}
\begin{defn}
A sequence $X$ is called \emph{$i$-discarded} if $d(X \uh n)=1$ for some $n$ where $\task(n)=i$.
\end{defn}
Note that a sequence $X\in\cs$ becomes $i$-discarded if there exists an initial segment $X \uh k$ such that the counter for task $i$ has reached the final value $1$ on $X \uh k$. By the conditions stated in Case 2 of the construction, below such an $X \uh k$ no further extra edges for task $i$ will branch off of $X$, hence the name ``discarded.''

\begin{lem}[Edge Existence Lemma]\label{lem:extension}
Assume that for $X\in\cs$ and for all $k\in\omega$  such that $\task(k)=i$ it holds that $X \uh k$ has an $i$-continuation and that $X$ is not $i$-discarded. Then $X$ contains an $i$-edge~$(\sigma,\tau)$; that is, $\sigma\prec \tau\prec X$.
\end{lem}
\begin{proof}
Assume that $X$ is not $i$-discarded. Let $m$ be maximal with $\task(m)=i$ and~${d(X \uh m)>0}$. 
We know by the following argument that $m$ is defined: First, an $m$ as described exists, since by \nameref{finite_i_edges}~\ref{finite_i_edges} we have that $w(i,q)$ is finite, and, by construction, $d(X \uh w(i,q))$~is set to a value strictly between $0$ and $1$. Secondly, by construction, any positive value $d(X \uh \ell)$ for some~$\ell > w(i,q)$ with~$\task(\ell)=i$ must be the result of a chain of $i$-edges branching off~$X$, as illustrated in Figure~\ref{edges_along_a_path}. 
Again by \nameref{finite_i_edges}~\ref{finite_i_edges}, $\X[i]$ is finite, and therefore any such chain can only have finite length, therefore only finitely many $\ell$ with $\task(\ell)=i$ can have~${d(X \uh \ell)>0}$. As a result a maximal~$m$ as described must exist. 

 Since, by assumption, $X\uh m$ has an $i$-continuation and $d(X \uh m) < 1$, the conditions of Subcase~2.1 of the construction are met. Therefore, eventually an $i$-edge of the form $(X \uh m, \tau)$ is attached at~$X \uh m$. By construction $d(\tau)=0$ and $d(\rho)\not= 0$ for all $\rho$ such that $X\uh m \prec \rho$, $\tau\not=\rho$, and $|\rho|=|\tau|$. By the choice of $m$ we must therefore have $\tau \prec X$.
\end{proof}

\begin{lem}[Continuity Lemma]
	\label{lemma_contin1} 
	The semi-measure $P$ has no atoms.
\end{lem}
\begin{proof}
	
	Note that by definition of the function $w$, there are no extra edges $(\sigma,\tau) \in \X$
	such that ${|\sigma|< w(i,q)\leq |\tau|}$ for any~$i$. That is, for any $i$, all flow that flows from nodes of length less than~$w(i,q)$ to nodes extending them flows through normal edges. Let $\sigma$ be a node of length~${w(i,q)-1}$. By construction $q(\sigma,\sigma0)=q(\sigma,\sigma1)\leq \nicefrac{1}{2}$, and hence, for $b\in\{0,1\}$,
	\[
	P(\sigma b)=R(\sigma  b)=q(\sigma,\sigma  b)\cdot R(\sigma) \leq \nicefrac{1}{2} \cdot P(\sigma).
	\]
	Since there are infinitely many numbers of the form $w(i,q)$, $i\in\omega$, we have $\lim_{n\rightarrow\infty}P(X\uh n)=0$ for every $X\in\cs$.
\end{proof}

\subsection{The roadmap}

Everything discussed thus far in this section forms the common part of the construction. In particular, we do not need to re-prove Lemma~\ref{fact_constr}, \nameref{finite_i_edges}~\ref{finite_i_edges},  \nameref{lem:extension}~\ref{lem:extension}, and \nameref{lemma_contin1}~\ref{lemma_contin1} for each application of the template. However, when applying the template to obtain different results, some parts of the construction need to be adapted to the statement that should be proved. There will still be a common structure with the following components.
\begin{description}
\item[Predicate $\emph{B}$] The predicate $B$ determines when edges are added to the digraph, and therefore the information that will be coded into the semi-measure constructed.
\item[Cut-off Lemma] Here we show that if any positive flow occurs beyond a node $\tau$, then at least some part of that flow must have passed through normal edges.
\item[Continuation Existence Lemma] To be able to apply the \nameref{lem:extension}~\ref{lem:extension} to all of the sequences in the support of the semi-measure we construct, we need to prove that the hypotheses of the lemma are satisfied by these sequences. That is, we need to prove that  for every $i\in\omega$, every sequence $X$ in the support is an $i$-continuation for all of its own initial segments~$X \uh n$ with $\task(n)=i$.
\item[Measure Lemma] This shows that the support of the constructed semi-measure $P$ has positive $\overline P$-measure. Note that, together with \nameref{lemma_contin1}~\ref{lemma_contin1} and using Proposition~\ref{prop_atoms_are_comp}, this implies that the support of~$P$ does not exclusively contain computable elements.  

\item[Verifying the desired properties] Finally we need to verify that the semi-measure we constructed has the desired properties needed for the statement that was to be shown.
\end{description}

\section{Implementing the Template}\label{sec-implementing}

\subsection{A first example} We begin by giving V'yugin's proof of Theorem \ref{thm-prob-alg}.

\theoremstyle{thm}
\newtheorem*{thm:thm-prob-alg}{Theorem \ref{thm-prob-alg}}
\begin{thm:thm-prob-alg}[V'yugin  \cite{Vyu12}]
	For any $\delta\in(0,1)$, there is a probabilistic algorithm that produces with probability at least $1-\delta$ a non-computable sequence that does not compute any Martin-L\"of random sequence.
\end{thm:thm-prob-alg}

To prove this, we will show the following more general statement.
\begin{samepage}
\begin{thm}[V'yugin \cite{Vyu12}]\label{thm:template-1}
For each $\delta\in(0,1)$, there is a left-c.e.\ semi-measure $P$ such that
\begin{itemize}
\item[(i)] $P$ has no atoms;
\item[(ii)] $\overline{P}(\cs)=\overline{P}(\supp{P})>1-\delta$; and
\item[(iii)] for each $X\in \supp{P}$ and each Turing functional $\Phi$, if $\Phi(X)$ is defined, then ${\Phi(X)\not\in\MLR}$.
\end{itemize}
\end{thm}
\end{samepage}

We obtain the desired probabilistic algorithm from Theorem \ref{thm:template-1} by applying Theorem \ref{thm-InduceSemiMeasures}~(ii):  Since $P$ is a left-c.e.\ semi-measure, there is some Turing functional $\Psi$ such that $P=\lambda_\Psi$.  The functional $\Psi$ equipped with a random oracle provides the probabilistic algorithmic satisfying the conditions of Theorem \ref{thm-prob-alg}.

One additional consequence of Theorem \ref{thm:template-1}  is that $\mathbf{r}\vee\mathbf{c}$ is not the top degree of $\D_\LV$, which we already showed via an alternative method in Section \ref{sec-lv-degrees}.  Indeed, since $\supp{P}$ contains no atoms and every atom of a left-c.e.\ semi-measure is computable, it follows that \[\overline{P}(\supp{P}\setminus \{X\colon X\text{ computable}\})>0.\] By Corollary~\ref{dfgsdafkdfjhsdsdfg}, this implies that $\supp{P}\setminus \{X\colon X\text{ computable}\}$ is non-negligible. But, by construction, the Levin-V'yugin degree generated by $\supp{P}\setminus \{X\colon X\text{ computable}\}$ is disjoint from $\mathbf{r}\vee\mathbf{c}$.

V'yugin originally proved this result in~\cite{Vyu76} without use of the machinery laid out in the previous section, but in a later article~\cite{Vyu12} he gave the proof discussed here.

To prove Theorem~\ref{thm:template-1}, we first need to specify the predicate $B$ and the function $\ct$, as in the template outlined above. For an elementary network $q'$ and nodes $\sigma$, $\tau$ with $\task(|\sigma|)=\task(|\tau|)$,   $B(q',\sigma,\tau)$ is defined to hold if and only if 
\begin{itemize}
\item[(a)] $\sigma\preceq \tau$,
\item[(b)] $d'(\tau\uh k)<1$ for all $k$ such that $1\leq k\leq |\tau|$, where $d'$ is the flow-delay function of~$q'$, and
\item[(c)] $\bigl|\Phi_{j,|\tau|}^\tau\bigr|>\lla \#(\sigma),s\rra$, where $\task(|\sigma|)=\lla j,s\rra$.
Here $\#(\sigma)$ denotes the position of $\sigma$ in the canonical lexicographic ordering of $\str$ and $\lla\cdot,\cdot\rra$ denotes a pairing function that satisfies $\lla m,n\rra\geq m+n$ for all $m,n \in \omega$.
\end{itemize}

The idea of this choice of~$B$ is that for each $i\in\omega$ such that $i=\lla j,s\rra$ for $j,s\in\omega$, we attach an $i$-edge between $i$-nodes $\sigma$ and $\tau$ only if 
 $\bigl|\Phi_{j,|\tau|}^\tau\bigr|>\lla\#(\sigma),s\rra$; that is, $\Phi_{j,|\tau|}^\tau$ is sufficiently long.  Moreover, we will ensure that for each $X\in\cs$, either there is some $n$ such that the flow out of $X\uh n$ is completely blocked, or, for each Turing functional $\Phi_j$ such that $\Phi_j(X)$ is defined, $\Phi_j(X)\notin\MLR$.  
 This latter condition will be accomplished by, for $\lla j,s\rra$-edges $(\sigma,\tau)$ with $s\in\omega$, enumerating $\tau$ into a Martin-L\"of test.

As for the choice of $\ct$, given $\delta\in(0,1)$, we let ${\ct(n)=(n+n_0)^2}$, where $n_0$~is such that 
\[
\sum_{n\in\omega}(n+n_0)^{-2}<\delta.
\] 
This will be used to prove  \nameref{lem:semimeasure-support}~\ref{lem:semimeasure-support} below.

Now let $P$ be the semi-measure produced by the template outlined in Section~\ref{subsec-template} when used with this specific choice of~$B$ and~$\ct$.
We establish that $P$ has the desired properties.

\begin{lem}[Measure Lemma]\label{lem:semimeasure-support}
$\overline{P}(\supp{P})>1-\delta$.
\end{lem}

For $X\in \supp{P}$ we already have that $P(X\uh n)>0$ for all $n$; that is, at any finite level $n$, not all measure has dissipated.  We will show that 
for all $n$, the amount of flow that flows into but not out of strings of length $n$ is bounded from above by 
$(n+n_0)^{-2}$. 
This implies that the total dissipation is  $\sum_n(n+n_0)^{-2}<\delta$, thus establishing the result.

In the construction, when an $i$-counter runs out along a path, the delay value is set to $1$ at some node $\sigma$ that is an initial segment of that path to remove the path from the support of the constructed semi-measure. As this means that {\em all} flow arriving in $\sigma$ is blocked at $\sigma$, the amount of measure lost this way could be very large. This is why we start the countdown with larger and larger numbers in the construction, as this ensures that there are more and more chances to add edges, which preserves more and more measure. 

On the other hand we {\em do} need that after finitely many attempts to add an edge we give up and block all flow along that path completely, as otherwise a single task might cause infinitely many of the failures described on page~\pageref{sdfsdsdfsdfsdgdhfgh}, which might prevent the construction from ever successfully handling the remaining tasks.  Furthermore, if a currently investigated functional~$\Phi_j$ stops producing output somewhere, then we only lose the measure currently delayed there; all the remaining measure keeps flowing through normal edges. The measure lost this way is another quantity that we need to control.

The trade-offs needed to reconcile these necessities make the construction quite complex and are the reason why establishing a lower bound for the remaining measure requires the following involved argument.

\begin{proof}[Proof of Lemma \ref{lem:semimeasure-support}]

By the definition of $R$ and $d$,
\begin{equation}\label{eq:sem-supp1}
\begin{split}
\sum_{|\sigma|=n+1}R(\sigma)&=\sum_{|\tau|=n}q(\tau,\tau0)R(\tau)+\sum_{|\tau|=n}q(\tau,\tau1)R(\tau)+\sum_{(\rho,\xi)\in \X, |\xi|=n+1}q(\rho,\xi)R(\rho)\\
&=\sum_{|\tau|=n}(1-d(\tau))R(\tau)+\sum_{(\rho,\xi)\in \X, |\xi|=n+1}q(\rho,\xi)R(\rho).
\end{split}
\end{equation}
\noindent
We set
\begin{equation}\label{eq:sem-supp2}
S_n=\sum_{|\sigma|=n}R(\sigma)-\sum_{(\rho,\xi)\in \X, |\xi|=n}q(\rho,\xi)R(\rho),
\end{equation}
so that it follows from (\ref{eq:sem-supp1}) and (\ref{eq:sem-supp2}) that
\begin{equation}\label{eq:sem-supp3}
S_{n+1}=\sum_{|\tau|=n}(1-d(\tau))R(\tau).
\end{equation}
That is, $S_{n+1}$ is the amount of flow into nodes of length $n+1$ that comes directly from nodes of length $n$ (and not through extra edges whose end nodes have length $n+1$).

We claim that $S_{n+1}\geq S_n-(n+n_0)^{-2}$ for all $n$.  For fixed $n$, we consider the possible values of $w(\task(n),q_{n-1})$. First, we consider Subcase~2.2 of the construction, where $w(\task(n),q_{n-1})<n$ but we added no extra edge $(\sigma,\tau)$ where $|\tau|=n$.  In this case, for each $\rho$ such that $|\rho|=n$, $d(\rho)=d_n(\rho)=d_{n-1}(\rho)=0$.  It then follows from (\ref{eq:sem-supp2}) and (\ref{eq:sem-supp3}) that $S_{n+1}= S_n$. 

Next, suppose that we are in Subcase~2.1 of the construction, where $w(\task(n),q_{n-1})<n$  and we added at least one extra edge $(\sigma,\tau)$ with $|\tau|=n$.  For $\sigma,\tau\in\str$, let
\[
\fan(\sigma,\tau)=\{\rho\colon |\rho|=|\tau|\;\wedge\; \sigma\prec \rho \;\wedge\;\rho\neq \tau\}.
\]
In Figures~\ref{fig_adding_edges} and~\ref{edges_along_a_path} the fans of extra edges were represented by dotted cones.

\begin{sublem}\label{lem-flowR}
For every $(\sigma,\tau)\in \X$,
\begin{equation}\label{sdfsdgsdsdjfgsdfhgsdfgwwejafsdhfas}
\sum_{\rho\in \fan(\sigma,\tau)} R(\rho)\leq(1-d(\sigma))R(\sigma).
\end{equation}
\end{sublem}
\begin{proof} The term on the left-hand side of the inequality is the total amount of flow that flows into all nodes in $\fan(\sigma,\tau)$, while the term on the right-hand side is the total flow into $\sigma$ (the node at the base of the fan) minus the flow that is diverted into the extra edge $(\sigma,\tau)$.  The only case where this inequality can fail to hold is if there is some flow through an extra edge 
$(\zeta,\xi)\in \X$ such that  $\zeta\prec\sigma\prec\xi\preceq\rho$ for some~$\rho\in\fan(\sigma,\tau)$.  However, since $(\sigma,\tau)\in\X$, the existence of such an extra edge $(\zeta,\xi)$ contradicts Lemma~\ref{fact_constr}.\footnote{For this, apply the lemma in such a way that the nodes in the current proof are  identified with the nodes of equal name appearing in the statement of the lemma.}  Thus the inequality must hold.
\renewcommand{\qedsymbol}{$\Diamond$}\end{proof}
The sum 
$
\sum_{|\rho|=n}d(\rho)R(\rho)
$
can be understood as the total amount of measure that is delayed at level $n$. Indeed, since $R(\rho)$ is the absolute amount of flow into a node $\rho$ and $d(\rho)$ is the relative fraction of flow delayed at $\rho$, we have that $d(\rho)R(\rho)$ is the absolute quantity of flow delayed at~$\rho$. 

Since we are in Subcase 2.1 (and therefore a non-trivial delay value at a node $\rho$ cannot be caused by an initial activation of $\rho$ but must be caused by an extra edge ending in a node $\tau$ with~${\rho\in \fan(\sigma,\tau)}$), we have:
\begin{align*}
\sum_{|\rho|=n}d(\rho)R(\rho)&=\sum_{(\sigma,\tau)\in \X,|\tau|=n}\;\sum_{\rho\in \fan(\sigma,\tau)} d(\rho)R(\rho)\\
\intertext{By definition of $d$ on $\rho\in\fan(\sigma,\tau)$: }
&=\sum_{(\sigma,\tau)\in \X,|\tau|=n}\frac{d(\sigma)}{1-d(\sigma)}\sum_{\rho\in \fan(\sigma,\tau)} R(\rho)\\
\intertext{By Sublemma \ref{lem-flowR}:}
&\leq \sum_{(\sigma,\tau)\in \X,|\tau|=n}d(\sigma)R(\sigma)\\
&=\sum_{(\sigma,\tau)\in \X,|\tau|=n}q(\sigma,\tau)R(\sigma).
\end{align*}

\smallskip

Then
\begin{equation}\label{eq-blah}
\begin{split}
S_{n+1}=\sum_{|\rho|=n}(1-d(\rho))R(\rho)&=\sum_{|\sigma|=n}R(\sigma)-\sum_{|\rho|=n}d(\rho)R(\rho)\\
&\geq\sum_{|\sigma|=n}R(\sigma)-\sum_{(\sigma,\tau)\in \X,|\tau|=n}q(\sigma,\tau)R(\sigma)=S_n.
\end{split}
\end{equation}
Lastly, in Case~1 of the construction, we have $w(\task(n),q_{n-1})=n$, and hence
\[
\sum_{|\rho|=n}d(\rho)R(\rho)\leq1/\ct(n)=(n+n_0)^{-2}.
\]
Consequently,
\begin{equation}
\begin{split}
S_{n+1}=\sum_{|\rho|=n}(1-d(\rho))R(\rho)&=\sum_{|\sigma|=n}R(\sigma)-\sum_{|\rho|=n}d(\rho)R(\rho)\\
&\geq\sum_{|\sigma|=n}R(\sigma)-(n+n_0)^{-2}\\
&\geq\sum_{|\sigma|=n}R(\sigma)-\sum_{(\sigma,\tau)\in \X,|\tau|=n}q(\sigma,\tau)R(\sigma)-(n+n_0)^{-2}\\
&=S_n-(n+n_0)^{-2}.
\end{split}
\end{equation}

Now since $S_{n+1}\geq S_n-(n+n_0)^{-2}$ for every $n$ and $S_0=1$, we have
\[
S_n\geq1-\sum_{i=1}^\infty(i+n_0)^{-2}>1-\delta.
\]
Lastly, by the definition of the support of a semi-measure, we have
\[
\overline{P}(\supp{P})=\inf_n\sum_{|\rho|=n} P(\rho)\geq \inf_n\sum_{|\rho|=n} R(\rho)\geq \inf_n S_n>1-\delta.\qedhere
\]
\end{proof}

\begin{lem}[Cut-off Lemma]\label{lem:blocked-flow}
	For $\tau\in\str$, $P(\tau)=0$ if and only if there is some $\sigma\prec \rho\preceq \tau$ such that $\rho \in\{\sigma0, \sigma1\}$ and $q(\sigma,\rho)=0$.
\end{lem}

\begin{proof}
	Assume that, for all $0\leq i<|\tau|$, $q\bigl(\tau\uh i, \tau\uh (i+1)\bigr)>0$ holds. Then by definition of $R$ we have 
	\[
	R(\tau) \geq \prod_{i=0}^{|\tau|-1} q\bigl(\tau\uh i, \tau\uh (i+1)\bigr) > 0,
	\]
	which together with $P(\tau)\geq R(\tau)$ implies $P(\tau)>0$.
	
	For the other direction, suppose there is some $n<|\tau|$ such that $q\bigl(\tau\uh n,\tau\uh(n+1)\bigr)=0$, but~${P(\tau)\neq 0}$.  
	Then there must be some extra edge $(\sigma,\rho)$ such that $\sigma\preceq \tau\uh n$ and $\tau\uh (n+1)\preceq \rho$. 
	We have that $q\bigl(\tau\uh n,\tau\uh(n+1)\bigr)=0$ implies $d(\tau\uh n)=1$.
	But, by condition~(b) in the definition of $B$ above, $(\sigma,\rho)$~can only be added if $d(\rho\uh k)<1$ for all $k$ such that $1\leq k\leq |\rho|$, contradicting the fact that $d(\rho\uh n)=d(\tau\uh n)=1$.
\end{proof}

\begin{lem}[Continuation Existence Lemma]\label{lem:edge-guarantees}
For every Turing functional $\Phi_j$, every ${X\in \supp{P}}$ such that $\Phi_j(X)$ is defined, and every $i=\lla j,s\rra$ for $s\in\omega$, $X$ is an $i$-continuation of~$X\uh m$ for every~$m\in\omega$ such that $\task(m)=i$.
\end{lem}

\begin{proof}
Fix $j,m,s\in\omega$, and let $i=\lla j,s\rra$.  Recall that $X$ is an $i$-continuation of $\sigma\in\str$ with~${\task(|\sigma|)=i}$ if $\sigma\prec X$ and $B(q_{n-1},\sigma,X\uh n)$ holds for almost all $n$  such that $\task(n)=i$.
Thus, to show that $X$ is an $i$-continuation of~$X\uh m$, it suffices to show that, for almost every $n$, the following two conditions from the definition of the predicate $B$ hold:
\begin{itemize}
\item[(b)] $d(X\uh k)<1$ for every $k$ such that $1\leq k\leq n$, and 
\item[(c)] $\bigl|\Phi_{j,n}^{X\uh n}\bigr|>\lla \#(X\uh m),s\rra$. 
\end{itemize}
Since $X\in \supp{P}$, $P(X\uh n)>0$ for every $n$, and it follows from the \nameref{lem:blocked-flow}~\ref{lem:blocked-flow} that $d(X\uh n)<1$ for every $n$, and so (b) holds.  Moreover, as $\Phi_j(X)$ is defined, for each $N\in\omega$, $|\Phi_{j,n}^{X\uh n}|\geq N$ for all sufficiently large $n$; thus, (c) holds.
\end{proof}

\begin{lem}\label{mlr_final_conclusion}
For any $X\in \supp{P}$ and any Turing functional $\Phi_j$ such that $\Phi_j(X)$ is defined, \[{\Phi_j(X)\notin\MLR}.\]
\end{lem}

\begin{proof}
For $s\in\omega$, let
	\[
	\U_s=
	\bigcup_{n\colon\task(n)=\lla j,s\rra}\;\bigcup_{\sigma\in \C_n}
	\;\;\;\llbracket\Phi_{j,n}^{\beta(\sigma,q_{n-1},n)}\rrbracket
	\]
where $\C_n$ is the set of the same name that was defined during the construction.	
	
Fix $s \in \omega$. Since $X\in \supp{P}$ and $\Phi_j(X)$ is defined, by \nameref{lem:edge-guarantees}~\ref{lem:edge-guarantees}, $X$ is an $i$-continuation of $X\uh m$ for $i=\lla j,s\rra$ and every $m\in\omega$  such that $\task(m)=i$.
Since~${X\in \supp{P}}$, $X$~cannot be $i$-discarded.
Then, by \nameref{lem:extension}~\ref{lem:extension}, there are $n,m\in\omega$ with $m<n$ such that there is an extra $i$-edge $(X\uh m,\beta(X\uh m,q_{n-1},n))$ such that $\beta(X\uh m,q_{n-1},n)=X\uh n$. 
It follows that $\llbracket\Phi_j^{X\uh n}\rrbracket$ is enumerated into $\U_s$.

Since $\bigl|\Phi_{j,n}^{\beta(\sigma,q_{n-1},n)}\bigr|>\lla \#(\sigma),s\rra$ for each $n \in \omega$ and $\sigma\in \C_n$,
\[
\lambda(\U_s)\leq\sum_{n\colon\task(n)=\lla j,s\rra}\;\sum_{\sigma\in \C_n}2^{-\lla \#(\sigma),s\rra}\leq 2^{-s}.
\]
Hence, $(\U_s)_{s \in \omega}$ is a Martin-L\"of test covering $\Phi_j(X)$, and thus ${\Phi_j(X)\notin\MLR}$.
\end{proof}

This completes the proof of Theorem~\ref{thm:template-1}, as \nameref{lemma_contin1}~\ref{lemma_contin1} establishes the Theorem's condition~(i), \nameref{lem:semimeasure-support}~\ref{lem:semimeasure-support} establishes condition~(ii), and Lemma~\ref{mlr_final_conclusion} establishes condition~(iii).

\medskip

In light of the second paragraph of the proof of Lemma \ref{mlr_final_conclusion} we can now formulate an intuitive understanding of  \nameref{lem:extension}~\ref{lem:extension}: It states that every path (that meets the conditions in the statement of the lemma) will eventually either be removed from the support of the semi-measure during its construction, or, if not, will be treated using the predicate $B$ to make sure all paths that remain in the support have the desired properties. In either case, the construction succeeds.

\subsection{A new application of the technique}

We now turn to the proof of Theorem \ref{thm-prob-alg2}, an extension of V'yugin's Theorem \ref{thm-prob-alg}.

\theoremstyle{thm}
\newtheorem*{thm:thm-prob-alg2}{Theorem \ref{thm-prob-alg2}}
\begin{thm:thm-prob-alg2}
For any $\delta\in(0,1)$, there is a probabilistic algorithm  that produces  with probability at least~$1-\delta$  a non-computable sequence that is not of DNC degree.
\end{thm:thm-prob-alg2}

To prove Theorem \ref{thm-prob-alg2}, we prove a strengthening of Theorem \ref{thm:template-1} in terms of a family of weak notions of randomness; just as Theorem \ref{thm-prob-alg} follows from Theorem \ref{thm:template-1}, so too will Theorem \ref{thm-prob-alg2} follow from this strengthening.  The following notion was explicitly defined by Higuchi~et~al.~\cite{HigHudSim14} and was further studied by Simpson and Stephan~\cite{SimSte15}.

\begin{defn}
Let $f\!\colon\str\rightarrow\omega$ be a total computable function.
\begin{itemize}
\item[(i)] An \emph{$f$\!-Martin-L\"of test} is a sequence of uniformly c.e.\ sets of strings $(U_i)_{i\in\omega}$ such that 
\[
\sum_{\sigma\in U_i}2^{-f(\sigma)}\leq 2^{-i}
\]
for every $i\in\omega$.
\item[(ii)] A sequence $X\in\cs$ is \emph{$f$\!-random} if $X\notin\bigcap_{i\in\omega}\llb U_i\rrb$ for every $f$\!-Martin-L\"of test $(U_i)_{i\in\omega}$.
\end{itemize}
\end{defn}

We will focus our attention on notions of $f$\!-randomness for sequences $X$ and functions~$f$ where $f$~is \emph{unbounded along} $X$; that is, $\lim_{n\rightarrow\infty}f(X\uh n)=\infty$.  We can now state our generalization of Theorem \ref{thm:template-1}.

\begin{samepage}
\begin{thm}\label{thm:template-2}
For each $\delta\in(0,1)$, there is a left-c.e.\ semi-measure $P$ such that
\begin{itemize}
\item[(i)] $P$ has no atoms;
\item[(ii)] $\overline{P}(\cs)=\overline{P}(\supp{P})>1-\delta$; and
\item[(iii)] for each $X\in \supp{P}$ and each Turing functional $\Phi$, if $\Phi(X)$ is defined, then $\Phi(X)$~is not $f$\!-random  for any computable function $f$ that is unbounded along $\Phi(X)$.
\end{itemize}
\end{thm}
\end{samepage}

We call a function $f\!\colon\str\rightarrow\omega$ {\em monotone} if for any $\sigma,\tau\in\str$ with $\sigma\preceq\tau$ we have that~${f(\sigma)\leq f(\tau)}$. 
For any~$f\!\colon\str\rightarrow\omega$, we define $f^*\!\colon\str\rightarrow\omega$ by letting, for each $\sigma \in \str$,
\[f^*(\sigma)=\max\{f(\tau)\colon \tau\preceq \sigma\}.\] 
Clearly, $f^*$ is monotone and we have $f(\sigma)\leq f^*(\sigma)$ for all $\sigma\in\str$. If $f$ is furthermore computable and unbounded along some $X\in\cs$, then the same holds for~$f^*$.
The proof of Theorem~\ref{thm:template-2} below will only ensure that~(iii)~holds for monotone~$f$. The following argument establishes that this is sufficient.
\begin{lem}\label{lem-monotone}
Let $f\!\colon\str\rightarrow\omega$ be a total computable function.  Then $X\in\cs$ is $f$\!-random if and only if $X$ is $f^*$-random.
\end{lem}

\begin{proof}
($\Leftarrow$:) Suppose that $X\in\cs$ is not $f$\!-random.  Then there is some $f$\!-Martin-L\"of test $(U_i)_{i\in\omega}$ such that $X\in\bigcap_{i\in\omega}\llb U_i\rrb$.  We claim that $(U_i)_{i\in\omega}$ is an $f^*$-Martin-L\"of test.  Indeed, since $f(\sigma)\leq f^*(\sigma)$ for all $\sigma\in\str$,
\[
\sum_{\sigma\in U_i}2^{-f^*(\sigma)}\leq \sum_{\sigma\in U_i}2^{-f(\sigma)}\leq 2^{-i}.
\]
It thus follows that $X$ is not $f^*$-random.

($\Rightarrow$:) Now suppose that  $X$ is not $f^*$-random. Then there is some $f^*$-Martin-L\"of test $(U_i)_{i\in\omega}$ such that $X\in\bigcap_{i\in\omega}\llb U_i\rrb$.  We modify $(U_i)_{i\in\omega}$ to produce an $f$\!-Martin-L\"of test covering $X$ as follows.  First note that for every $\sigma\in\str$, if $f(\sigma)\neq f^*(\sigma)$, then there is some $\tau\prec\sigma$ such that $f(\tau)=f^*(\tau)=f^*(\sigma)$.  In this case, let us set $\widehat\sigma=\tau$.  In the case that $f(\sigma)=f^*(\sigma)$, set $\widehat\sigma=\sigma$; in either case, we have $\widehat\sigma\preceq\sigma$.  Then for each $i\in\omega$ and $\sigma \in \str$, we define $\widehat {U}_i$ so that $\widehat\sigma\in \widehat {U}_i$ if and only if $\sigma\in U_i$.  It follows that $(\widehat{U}_i)_{i\in\omega}$ is an $f$\!-Martin-L\"of test, since
\[
\sum_{\widehat\sigma\in \widehat{U}_i}2^{-f(\widehat\sigma)}= \sum_{\sigma\in U_i}2^{-f^*(\sigma)}\leq 2^{-i}.
\]
Next, since for every $\sigma\in\str$ we have $\widehat\sigma\preceq\sigma$, it follows that $\llb U_i \rrb \subseteq \llb\widehat{U}_i \rrb$ for every $i\in\omega$, and hence $X\in\bigcap_{i\in\omega}\llb U_i\rrb\subseteq\bigcap_{i\in\omega}\llb \widehat{U}_i\rrb$. We thus conclude that $X$ is not $f$\!-random.
\end{proof}

The general strategy of the proof of Theorem \ref{thm:template-2} is much like that of the proof of Theorem~\ref{thm:template-1}, but with several modifications.  First, since we want that elements of $\supp{P}$ cannot compute any $f$\!-random sequences for any monotone unbounded computable $f\!\colon\str\rightarrow\omega$, our construction will have to involve all total computable functions.  Of course, there is no effective enumeration of such functions, so we have to work with an enumeration of all partial computable functions $(\phi_e)_{e\in\omega}$ (where each $\phi_e$ is viewed as a map from $\str$ to $\omega$).  Moreover, we can assume that all functions~$\phi_e$ are monotone simply by replacing each~$\phi_e$ with the corresponding monotone~$\phi_e^*$.

Second, we have to modify the definition of the predicate $B$ from the proof of Theorem~\ref{thm:template-1} as follows:
For an elementary network $q'$ and nodes $\sigma$,$\tau$ with $\task(|\sigma|)=\task(|\tau|)$, $B(q',\sigma,\tau)$ is defined to hold if and only if
\begin{itemize}
\item[(a)] $\sigma\preceq \tau$,
\item[(b)] $d'(\tau\uh k)<1$ for all $k$ such that $1\leq k\leq |\tau|$, where $d'$ is the flow-delay function of~$q'$, and
\item[(c$^*$)] there is some $\rho\preceq\Phi_{j,|\tau|}^\tau$ such that $\phi_{e,|\tau|}(\rho)\halts
>\lla \#(\sigma),s\rra$, where $\task(|\sigma|)=\lla j,s,e\rra$.\footnote{We cannot simply let $\rho=\Phi_{j,|\tau|}^\tau$ as there is no guarantee that running $\phi_e$ with input $\Phi_{j,|\tau|}^\tau$ terminates within $|\tau|$~steps.} 
\end{itemize}

Observe that for non-total $\phi_e$ condition (c$^*$) may never become true and as a result we may never attach a $\task(|\sigma|)$-edge to~$\sigma$. This is not a problem, as condition (iii) in Theorem~\ref{thm:template-2} only makes a promise about total functions, so no action is required in this case. Tasks of the form~$\lla j,\cdot,e\rra$ can therefore be safely ignored when verifying that the construction yields the desired semi-measure.

As in the previous subsection, that \nameref{lemma_contin1}~\ref{lemma_contin1} holds is an inherent feature of the construction technique, independently of the specific choice of the predicate $B$ and the countdown function~$\ct$. 
\nameref{lem:semimeasure-support}~\ref{lem:semimeasure-support} also still holds since its truth does not depend on the specific choice of~$B$ while $\ct$~is unchanged. As for  \nameref{lem:blocked-flow}~\ref{lem:blocked-flow}, an inspection of its proof shows that it only relies on condition~(b) of the predicate~$B$ which we haven't changed from the last subsection; so this lemma still holds as well. The Continuation Existence Lemma, however, needs to be modified. 

\begin{lem}[Modified Continuation Existence Lemma]\label{lem:mod-edge-guarantees}
Suppose that we have a Turing functional~$\Phi_j$, some $X\in \supp{P}\cap\dom(\Phi_j)$, and a monotone total computable function $\phi_e$ such that $\phi_e(\Phi_j(X)\uh n)$ is unbounded in $n$. Then for every $i$ of the form $\lla j,s,e\rra$ for some~${s\in\omega}$, $X$ is an $i$-continuation of~$X\uh m$ for every~$m\in\omega$  such that $\task(m)=i$.
\end{lem}
\begin{proof}
That condition~(b) in the predicate~$B$ holds is shown as in the proof of Lemma~\ref{lem:edge-guarantees}. 

For condition~(c), since $X\in\dom(\Phi_j)$ and $\phi_e(\Phi_j(X)\uh n)$ is unbounded in $n$, there are~$n_1$ and~$n_2$ such that $\phi_{e,n_2}(\Phi_j^{X\uh n_1}) {\downarrow}> \lla \#(X\uh m),s\rra$. Then for any $n \geq \max\{n_1,n_2\}$, $\Phi_j^{X\uh n}$ has the initial segment~$\rho=\Phi_j^{X\uh n_1}$ such that 
$\phi_{e,n}(\rho) = \phi_{e,n_2}(\rho) > \lla \#(X\uh m),s\rra$.
Therefore, condition~(c) holds for any large enough~$n$ with $\task(n)=i$.
\end{proof}

Lastly, we prove the following.

\begin{lem}\label{t-random_final_conclusion}
Let $f\!\colon\str\rightarrow\omega$ be a monotone total computable function.  For any~${X\in \supp{P}}$ and any Turing functional $\Phi$ such that $X\in\dom(\Phi)$, if $f(\Phi(X)\uh n)$ is unbounded in~$n$, then $\Phi(X)$ is not $f$\!-random.
\end{lem}

\begin{proof}
Let $e$ be the index of $f$ as a partial computable function and let $j$ be the index of~$\Phi$.  
Then we define an $f$\!-Martin-L\"of test $(U_s)_{s\in\omega}$, where $U_s$ consists of all strings of the form~$\Phi_{j,n}^{\beta(\sigma,\,q_{n-1},\,n)}$ where~$n\in\omega$, $\sigma\in \C_n$, and $\task(n)=\lla j,s,e\rra$. 

For each $X\in\supp{P}\cap\dom(\Phi_j)$ such that  $f(\Phi_j(X)\uh n)$ is unbounded in $n$, by the Modified Continuation Existence Lemma \ref{lem:mod-edge-guarantees}, $X$ is an $i$-continuation of $X\uh m$ for every~$i$ such that~${i=\lla j,s,e\rra}$ for some~$s\in\omega$ and every $m\in\omega$  such that $\task(m)=i$.
Furthermore, $X$ is not $i$\nobreakdash-discarded, and hence by \nameref{lem:extension}~\ref{lem:extension} there is an $i$-edge $\bigl(X\uh m,\beta(X\uh m,q_{n-1},n)\bigr)$ such that ${\beta(X\uh m,q_{n-1},n)=X\uh n}$ for some~$n,m\in\omega$ with~$m<n$.  Since $\task(n)=\lla j,s,e\rra$, it follows that $\Phi_{j,n}^{X\uh n}$~is enumerated into~$U_s$.

Choose any $\eta \in U_s$. Then, by definition of $U_s$, there is some $\tau$ such that $\Phi_{j,n}^\tau = \eta$ and some~$\sigma \in \C_{|\tau|}$ such that $(\sigma, \tau)\in \X$. By condition (c$^*$) of the predicate~$B$, the fact that this extra edge was added to the digraph implies that there is a witnessing initial segment $\rho \preceq \Phi_{j,n}^\tau$ such that 
\[\phi_e(\eta)=\phi_e( \Phi_{j,n}^{\tau})\geq 
\phi_{e}(\rho)
=\phi_{e,|\tau|}(\rho)
>\lla \#(\sigma),s\rra;\]
here the first inequality follows from the monotonicity of $\phi_e$.
As this line of reasoning applies to every $\eta \in U_s$, we obtain
 \[
 \sum_{\eta\in U_s}2^{-\phi_e(\eta)}\leq\sum_{n\colon\task(n)=\lla j,s,e\rra}\;\sum_{\sigma\in \C_n}2^{-\lla \#(\sigma),s\rra}\leq 2^{-s}.
 \] 
 Hence, $(U_s)_{s\in\omega}$ is an $f$\!-Martin-L\"of test covering $\Phi_j(X)$, and so $\Phi_j(X)$ is not $f$\!-random.  
\end{proof}
\noindent This completes the proof of Theorem~\ref{thm:template-2}.
We can recast this result in terms of autocomplexity as well as in terms of being of DNC degree.  Recall that the Kolmogorov complexity of a string~$\sigma\in\str$ is defined by 
$
K(\sigma)=\min\{|\tau|\colon U(\tau)\halts=\sigma\},
$
where $U$ is a universal, prefix-free Turing machine.  Moreover, a function $f\!\colon\omega\rightarrow\omega$ is called an \emph{order} if $f$ is unbounded and non-decreasing.

\begin{defn}
$X\in\cs$ is \emph{autocomplex} if there is an $X$-computable order $f$ such that ${K(X\uh n)\geq f(n)}$ for every $n\in\omega$.  
\end{defn}

The following two propositions provide alternative characterizations of $f$\!-randomness.

\begin{prop} [Higuchi, Hudelson, Simpson, Yokoyama~\cite{HigHudSim14}]
$X\in\cs$ is autocomplex if and only if there is some computable function $f\!\colon\str \rightarrow \omega$ such that $f$ is unbounded along $X$ and $X$~is $f$\!-random.
\end{prop}

\begin{prop}[Kjos-Hanssen, Merkle, Stephan~\cite{KjoMerSte11}]
$X\in\cs$ is autocomplex if and only if $X$ is of DNC degree.
\end{prop}

We can now recast Theorem \ref{thm:template-2} as follows.

\begin{corollary}\label{cor:template}
For each $\delta\in(0,1)$, there is a left-c.e.\ semi-measure $P$ such that
\begin{itemize}
\item[(i)] $P$ has no atoms;
\item[(ii)] $\overline{P}(\cs)=\overline{P}(\supp{P})>1-\delta$; and
\item[(iii)] for each $X\in \supp{P}$ and each Turing functional $\Phi$, if $\Phi(X)$ is defined, then $\Phi(X)$~is not autocomplex, or equivalently, $\Phi(X)$ is not of DNC degree. Equivalently, for each $X\in \supp{P}$, $X$~is not of DNC degree.
\end{itemize}
\end{corollary}

By the same argument as the one that immediately follows the statement of Theorem \ref{thm:template-1}, Corollary \ref{cor:template} yields an alternative proof of the fact that 
$\mathbf{d}\vee\mathbf{c}$ is not the top $\LV$-degree.

\section{Applications to $\Pi^0_1$ Classes}\label{sec-conclusion}

As we have seen, V'yugin's construction as laid out in Sections \ref{sec-building} and \ref{sec-implementing} yields significant results in the study of the $\LV$-degrees.  As noted in the introduction, the construction also has some interesting consequences for the study of $\Pi^0_1$ classes, that is, effectively closed subsets of $\cs$.  
In particular, for each of the semi-measures $P$ defined via V'yugin's construction, for~${\sigma\in\str}$, the condition $P(\sigma)=0$ is computable, as $P(\sigma)=0$ if and only if $q(\sigma)$ is set to 0 at stage $|\sigma|$ in the construction of $P$.  This implies that in each case, the support of~$P$ is a $\Pi^0_1$~class.  We thus establish the two corollaries stated in the introduction.

\begin{samepage}
\newtheorem*{thm:cor1}{Corollary \ref{cor1}}
\begin{thm:cor1}{\ }
For every $\delta\in(0,1)$, there is a Turing functional $\Phi$ such that 
\begin{itemize}
	\item[(i)] $\Phi$ maps no set of positive measure to any single sequence,
	
	\item [(ii)] the domain of $\Phi$ has Lebesgue measure at least $1-\delta$,
	\item [(iii)] the range of $\Phi$ is a $\Pi^0_1$ class, and
	\item [(iv)] no sequence in the range of $\Phi$ computes a Martin-L\"of random sequence.
\end{itemize}
\end{thm:cor1}
\end{samepage}
\begin{samepage}
\newtheorem*{thm:cor2}{Corollary \ref{cor2}}
\begin{thm:cor2}{\ }
For every $\delta\in(0,1)$, there is a Turing functional $\Phi$ such that 
\begin{itemize}
	\item[(i)]  $\Phi$ maps no set of positive measure to any single sequence,
	
	\item [(ii)] the domain of $\Phi$ has Lebesgue measure at least $1-\delta$,
	\item [(iii)] the range of $\Phi$ is a $\Pi^0_1$ class, and
	\item [(iv)] no sequence in the range of $\Phi$ is of DNC degree.
\end{itemize}
\end{thm:cor2}
\end{samepage}

\section*{Acknowledgments}

The authors would like to thank Laurent Bienvenu, Frank Stephan, Mushfeq Khan, and Paul Shafer for helpful discussions.  In particular, Bienvenu contributed to the reconstruction of the proof of Theorem~\ref{thm:rand-atom}, Khan informed us of the proof of the negligibility of the collection of non-computable sequences of hyperimmune-free degree in Proposition \ref{prop-negligible}, and Stephan pointed out the relation with immunity notions discussed in Corollary~\ref{biimmunedegree}. The authors would also like to thank the referees for detailed and insightful comments.

\medskip

H\"olzl was supported by a Feodor Lynen postdoctoral research fellowship by the Alexander von Humboldt Foundation
and by the Ministry of Education of Singapore through grant {R146\nobreakdash-000\nobreakdash-184\nobreakdash-112} (MOE2013\nobreakdash-T2\nobreakdash-1\nobreakdash-062). Porter was supported by the National Science Foundation through grant OISE-1159158 as part of the International Research Fellowship Program and by the John Templeton Foundation as part of the project ``Structure and Randomness in the Theory of Computation,'' and by the National Security Agency Mathematical Sciences Program grant H98230-I6-I-D310 as part of the Young Investigator's Program.

Both authors received travel support from the Bayerisch-Franz\"osisches Hochschulzentrum/\allowbreak Centre de Coop\'eration Universitaire
 Franco-Bavarois. In addition, H\"olzl received travel support from the above Templeton grant for a collaborative visit to work with Porter at the University of Florida.

  The opinions expressed in this publication are those of the authors and do not necessarily reflect the views of the John Templeton Foundation.

\bibliographystyle{alpha}
\bibliography{vyuginalgebra}

\end{document}